\documentclass[12pt]{amsart}
\usepackage[centering, margin={0.65in, 1in}, includeheadfoot]{geometry}

\usepackage{mathrsfs}
\usepackage{amsfonts}
\usepackage{tikz}
\usepackage{amssymb,bm}
\usepackage{amsmath}
\usepackage{amsthm}

\newtheorem{thm}{Theorem}[section]

\newtheorem{lem}[thm]{Lemma}
\newtheorem{cor}[thm]{Corollary}
\newtheorem{prop}[thm]{Proposition}
\newtheorem{rmk}[thm]{Remark}
\newtheorem{eg}[thm]{Example}

\newcommand{\Fn}{\mathbb{F}_{q^n}}
\newcommand{\F}{\mathbb{F}_{q}}
\newcommand{\Ln}{\mathscr{L}_n(\Fn)}

\DeclareMathOperator{\id}{id}
\DeclareMathOperator{\im}{im}
\DeclareMathOperator{\GL}{GL}

\begin{document}
 
\title{On the inverses of some classes of permutations of finite fields}
 
\author{Aleksandr Tuxanidy and Qiang Wang}

\address{School of Mathematics and Statistics, Carleton
University,
 1125 Colonel By Drive, Ottawa, Ontario, K1S 5B6,
Canada.} 

\email{AleksandrTuxanidyTor@cmail.carleton.ca, wang@math.carleton.ca}

 
\thanks{The research of Aleksandr Tuxanidy and Qiang Wang is partially supported by OGS and NSERC, respectively, of Canada.}
 
 \maketitle
 
 \begin{abstract}
We study the compositional inverses of some general classes of permutation polynomials over finite fields. 
 We show that we can write these inverses in terms of the inverses of two other polynomials bijecting subspaces of the finite field, where
 one of these is a linearized polynomial. In some cases we are able to explicitly obtain these inverses, thus obtaining the compositional inverse of the permutation in question.
 In addition we show how to compute a linearized polynomial inducing the inverse map over subspaces on which a prescribed linearized polynomial induces a bijection.
 We also obtain the explicit compositional inverses of two classes of permutation polynomials generalizing those whose compositional inverses were recently obtained in 
 \cite{wu_new} and \cite{wu}, respectively.
\end{abstract}

 \section{Introduction}
 Let $q = p^m$ be the power of a prime number $p$, let $\F$ be a finite field with $q$ elements, and let $\F[x]$ be the ring of polynomials over $\F$. 
 We denote the composition of two polynomials $f,g \in \F[x]$ by $(f \circ g) (x) := f(g(x))$. We call $f \in \F[x]$ a {\em permutation polynomial} of $\F$ 
 if it induces a permutation of $\F$. Note that since $x^q = x$ for all $x \in \F$, one only needs to consider polynomials of degree less than $q$.
 It is clear that permutation polynomials form a group under composition and reduction modulo $x^q - x$ that is isomorphic to the symmetric group on $q$ letters. Thus for any permutation polynomial $f \in \F[x]$,
 there exists a unique $f^{-1} \in \F[x]$ such that $f^{-1}(f(x)) \equiv f(f^{-1}(x)) \equiv x \pmod{x^q - x}$. Here $f^{-1}$ is defined as the {\em compositional inverse} of $f$,
 although we may simply call it sometimes the {\em inverse} of $f$ on $\F$. When we think of $x \in \F$ 
 as fixed and view $f \in \F[x]$ as a mapping of $\F$, we call $f^{-1}(x)$ the {\em preimage} of $x$ under $f$.
 
 The construction of permutation polynomials over finite fields is an old and difficult subject that continues to attract interest due to their 
 applications in cryptography \cite{rivest}, \cite{schwenk},
 coding theory \cite{ding13}, \cite{chapuy}, and combinatorics \cite{ding}. 
 See also \cite{akbary0}, \cite{akbary}, \cite{akbary2}, \cite{charpin}, \cite{coulter}, \cite{fernando}, \cite{hou}, 
 \cite{kyureghyan}, \cite{marcos}, \cite{mullen}, \cite{wang13}, \cite{yuan}, \cite{zha}, \cite{zieve}, and the references therein for some recent work in the area. However, the problem of 
 determining the compositional inverse of a permutation polynomial seems to be an even more complicated problem. 
 In fact, there are very few known permutation polynomials whose
 explicit compositional inverses have been obtained, and the resulting expressions are usually of a complicated nature except for the classes of the 
 permutation linear polynomials, 
 monomials, Dickson
 polynomials. In addition, see \cite{muratovic} and \cite{wang} for the characterization of the inverse of permutations of $\F$ 
 with form $x^rf(x^s)$ where $s \mid (q-1)$. The following list of permutation polynomials, whose compositional inverses have been found more recently, was given in \cite{wu}.
 Denote by $T_{q^n|q}$ the 
 {\em trace} (linearized) polynomial defined by $T_{q^n|q}(x) = \sum_{i=0}^{n-1}x^{q^i}$, which induces a surjection of $\Fn$ onto $\F$. When it will not cause confusion, we abbreviate this with $T$.
 
 (1) Linearized polynomials over $\Fn$, also called {\em $q$-polynomials}, which are $\F$-linear maps when seen as transformations of $\Fn$. These are the polynomials with form $L(x) := \sum_{i=0}^{n-1} a_i x^{q^i}$. 
 Their compositional inverse was given in Theorem 4.8, \cite{wu_lin}, in terms of cofactors of elements in the first column of the associate {\em Dickson matrix} given by
  $$
  D_L =  \left( \begin{array}{cccc}
                 a_0 & a_1 &\cdots & a_{n-1}\\
                 a_{n-1}^q & a_0^q & \cdots & a_{n-2}^q\\
                 \vdots & \vdots & & \vdots\\
                 a_1^{q^{n-1}} & a_2^{q^{n-1}} & \cdots & a_0^{q^{n-1}}
                \end{array} \right). 
$$
 This result was recently employed in \cite{wu_new} to explicitly obtain the compositional inverse of the class of linearized permutation polynomials
 $x^2 + x + T_{2^n|2}(x/a)$ over $\mathbb{F}_{2^n}$, $n$ odd, where $a \in \mathbb{F}_{2^n}^*$ is such that $T_{2^n|2}(a^{-1}) = 1$.
 
 Note that $L$ is a permutation polynomial of $\Fn$ if and only if $D_L$ is non-singular \cite{lidl}.
 In fact, the algebra of the $q$-polynomials over $\Fn$, denoted $\Ln$, is isomorphic to the algebras of the $n \times n$ Dickson matrices over $\Fn$, 
 and $n \times n$ matrices over $\F$ \cite{wu_lin}. 
 
 (2) The bilinear polynomial $x(T_{q^n|q}(x) + ax)$ over $\Fn$ where $q$ is even, $n$ is odd, and $a \in \F \setminus\{0,1 \}$. Its compositional inverse was obtained in \cite{coulter1}. 
 
 (3) The bilinear polynomial $f(x) = x (L(T_{q^n|q}(x)) + a T_{q^n|q}(x) + a x )$ where $q$ is even, $n$ is odd, $L \in \F[x]$ is a $2$-polynomial, and $a \in \F^{*}$. 
 Its compositional inverse was given
 in \cite{wu} in terms of the inverse of $xL(x)$ when restricted to $\F$. 
 
 Before we delve any further, we fix some notations. If we view $f \in \F[x]$ as a map of $\F$ and are given a subspace $V$ of $\F$, 
 we denote by $f|_V$ the map obtained by restricting $f$ to $V$. If $f$ is a bijection
 from $V$ to a subspace $W$ of $\F$, we mean by $f^{-1}|_{W}$ the inverse map of $f|_V$. 
 When the context is clear we may however mean by $f^{-1}|_{W}$ a polynomial in $\F[x]$ inducing the
 inverse map of $f|_V$. In this case we call $f^{-1}|_W$ the {\em compositional inverse of $f$ over $V$}.
 In both of the following cases, if $f(x)$ is viewed as an element of $\F$, or a polynomial in $\F[x]$, we denote $1/f(x) := f(x)^{q-2}$. Thus for instance $1/0 = 0$ with our notation.
 
 The main method used in \cite{wu} to obtain the compositional inverse of the permutation in (3) was
 to decompose the finite field $\Fn$ into the direct sum $\F \oplus \ker(T)$ thereby converting the 
 problem of computing the inverse of $f$ into the problem of computing the inverse of a bivariate function 
 permuting $\F \oplus \ker(T)$. This in turn is equivalent to obtaining the inverses of 
 two permutation polynomials, $\bar{f}, \varphi$, over the subspaces $\F$ and $\ker(T)$, respectively, where $\varphi$ is
 a $2$-polynomial (in fact, quadratic). As $q$ is even and $n$ is odd, $T$ is idempotent, i.e., $T \circ T = T$,
 and thus the (bijective) transformation $\phi: \Fn \rightarrow \F \oplus \ker(T)$ can be defined by
 $\phi(x) = (T(x), x - T(x))$, since $T(\Fn) = \F$ additionally. Because computing the inverse of any $p$-polynomial permutation
 on a subspace of $\Fn$ amounts to solving a system of linear equations, this problem is thus finally reduced to 
 computing the inverse of $\bar{f}|_{\F}$. 
 
 In this paper we extend this useful method to other more general classes of permutation polynomials
 given in \cite{akbary}, and \cite{yuan}, particularly written in terms of arbitrary linearized polynomials $\psi$ (instead of just $T$). 
 We write their inverses 
 in terms of the inverses of two other polynomials, one of these being linearized, over subspaces. In some special cases we can determine these inverses, thus obtaining
 the full result. However, given that we may not always have a ``nice enough'' expression for $\ker(\psi)$,
 we instead use the map $\phi_\psi : \Fn \rightarrow \psi(\Fn) \oplus S_\psi$, where 
 $S_{\psi} := \{x - \psi(x) \mid x \in \Fn  \}$ -- a subspace of $\Fn$, in similarity with $\ker(T)$ as above -- defined by
 $\phi_\psi(x) = (\psi(x), x - \psi(x))$.
  In the case that $\psi$ induces an idempotent map of $\Fn$, $S_\psi = \ker(\psi)$. 
 As $\phi_\psi$ is injective for any such $\psi$ (but not surjective in the case that
 $\psi$ is not idempotent), we similarly transform the problem of computing the inverses of 
 permutation polynomials into the problem of computing the inverses of two other bijections,
 $\bar{f}, \varphi$, over the subspaces $\psi(\Fn)$ and $S_\psi$, respectively. 
 
 Many of our resulting expressions for compositional inverses, or preimages in some cases, 
 are particularly written in terms of linearized polynomials inducing the inverse map over subspaces on which prescribed 
 linearized polynomials induce a bijection. Nonetheless, in Theorem \ref{thm:LinearizedInverse} we show how to obtain such compositional inverses on subspaces. 
 As we show there, obtaining the coefficients
 of such linearized polynomials is equivalent to solving a system of linear equations. 
 However in order to set up such a system one is required to first obtain a linearized polynomial, $K \in \Ln$, inducing an idempotent map with a prescribed kernel (we can always do this).
 But clearly, if we are given an idempotent $\psi \in \Ln$, i.e., $S_\psi = \ker(\psi)$, and $S_\psi$ is such a prescribed kernel, then one can simply let $K = \psi$.
 In the simultaneous case that the characteristic $p$ does not divide $n$, the coefficients of $\varphi \in \Ln$ belong to $\F$ (i.e., the associate Dickson matrix is {\em circulant}), and $\varphi$ induces a bijection between two $\F$-subspaces, we can 
 quickly solve the corresponding linear system by using the Circular Convolution Theorem, for the Discrete Fourier Transform (DFT), together with a
 Fast Fourier Transform (FFT). See \cite{jones} for details regarding circulants and their close relation to the DFT.
 
 As an example of our results, we obtain in Theorem \ref{inv0} the compositional inverse of the permutation $f$
 in the following theorem, under the assumption that $\varphi$ bijects $S_\psi$, and written in terms of the inverses of $\bar{f}|_{\psi(\Fn)}$ and $\varphi|_{S_\psi}$.
 
 \begin{thm}[{\bf Theorem 5.1, \cite{akbary}}]\label{thm0}
  Consider any polynomial $g \in \Fn[x]$, any additive polynomials $\varphi,\psi \in \Fn[x]$, any $q$-polynomial $\bar{\psi} \in \Fn[x]$ satisfying $\varphi \circ \psi = \bar{\psi} \circ \varphi$
 and $|\psi(\Fn)| = |\bar{\psi}(\Fn)|$, and any polynomial $h \in \Fn[x]$ such that $h(\psi(\Fn)) \subseteq \F \setminus \{0\}$. Then
 $$
f(x) = h(\psi(x))\varphi(x) + g(\psi(x))
$$
permutes $\mathbb{F}_{q^n}$ if and only if

(i) $\ker (\varphi) \cap \ker (\psi) = \{0\}$; and

(ii) $\bar{f}(x) := h(x) \varphi(x) + \bar{\psi}(g(x))$ is a bijection from $\psi(\mathbb{F}_{q^n})$ to $\bar{\psi}(\Fn)$.
 \end{thm}
 
 \begin{thm}\label{inv0}
  Using the same notations and assumptions of Theorem \ref{thm0}, assume that $f$ is a permutation of $\Fn$, and further assume
  that $|S_\psi| = |S_{\bar{\psi}}|$ and $\ker(\varphi) \cap \psi(S_\psi) = \{0\}$. Then $\varphi$ induces a bijection from $S_\psi$ to $S_{\bar{\psi}}$. Let $\bar{f}^{-1}, \varphi^{-1}|_{S_{\bar{\psi}}} \in \Fn[x]$
  induce the inverses of $\bar{f}|_{\psi(\Fn)}$ and $\varphi|_{S_{\psi}}$, respectively. Then the compositional inverse of $f$ on $\Fn$ is given by
 $$
  f^{-1}(x) = \bar{f}^{-1}\left(\bar{\psi}(x)\right) + \varphi^{-1}|_{S_{\bar{\psi}}}   \left( \dfrac{x - \bar{\psi}(x) - g\left(\bar{f}^{-1}\left(\bar{\psi}(x)\right)\right) + \bar{\psi}\left(g\left(\bar{f}^{-1}\left(\bar{\psi}(x)\right)\right)\right) }     {h\left(\bar{f}^{-1}\left(\bar{\psi}(x)\right)\right)}   \right).
  $$
  Furthermore, if $\varphi$ induces a bijection from $\psi(\Fn)$ to $\bar{\psi}(\Fn)$, then $\varphi$ is a permutation of $\Fn$ and the compositional inverse of $f$ on $\Fn$ is given by 
  $$
  f^{-1}(x) = \varphi^{-1}\left( \dfrac{x - g\left( \bar{f}^{-1}(\bar{\psi}(x)) \right)}{h\left( \bar{f}^{-1}(\bar{\psi}(x))  \right)}   \right).
  $$
 \end{thm}
 
 Note that condition (i) and the fact that the images of $\psi$, $\bar{\psi}$, are equally sized in Theorem \ref{thm0} imply that $\varphi$ is a bijection from $\ker(\psi)$
 to $\ker(\bar{\psi})$, for instance not necessarily satisfying $\ker(\varphi) \cap S_{\psi} = \{0\}$ (even though $\ker(\psi) \subseteq S_\psi$) required 
 to attain injectivity on $S_\psi$, which is an imposed hypothesis of Theorem \ref{inv0}. Thus our result is further restricted by our imposed assumption that
 $\varphi$ bijects $S_\psi$, a superset of $\ker(\psi)$. 
 However, as previously described, in the case that $\psi$ is idempotent or $\varphi$ is a permutation
 of $\Fn$ (see (1) above), we can get the inverse of $\varphi$ on $S_\psi$, and hence obtain
 the compositional inverse of $f$ in terms of a polynomial inducing the inverse map of $\bar{f}|_{\psi(\Fn)}$. Note that one may find several linearized polynomials
 inducing idempotent endomorphisms of $\Fn$. See Remark \ref{rmk:num_idempotent} for more details regarding this.
 
 As a consequence of this result we obtain several corollaries, like the following,
 in terms of the inverse of $\bar{f}|_{\psi(\Fn)}$ and the inverse of a linearized permutation polynomial of $\Fn$, which is already known (see (1) above). 
 
 \begin{rmk}
 To clarify an ambiguity in the presentation of our results: In many of the following results we make citations in the style of ``See Theorem `x'...''
 or ``See Corollary `x'...''
 to refer to a result in another paper where the construction of the considered
  permutation polynomial was obtained. However the compositional inverses given here in the statements of these results are ours. 
 \end{rmk}

 \begin{cor}[See Theorem 5.1 (c), \cite{akbary}]\label{zero}
  Consider $q$-polynomials $\varphi,\psi, \bar{\psi} \in \Fn[x]$ satisfying $\varphi \circ \psi = \bar{\psi} \circ \varphi$
 and $|\psi(\Fn)| = |\bar{\psi}(\Fn)|$. Let $g, h \in \Fn[x]$ be such that $(\bar{\psi} \circ g)|_{\psi(\Fn)} = 0$ and $h(\psi(\Fn)) \subseteq \F \setminus \{0\}$. 
 If
 $$
f(x) = h(\psi(x))\varphi(x) + g(\psi(x))
$$
permutes $\mathbb{F}_{q^n}$, then $\varphi$ permutes $\Fn$ as well, and the inverse of $f$ on $\Fn$ is given by
$$
f^{-1}(x) = \dfrac{\varphi^{-1}\left(x - g\left( \dfrac{\varphi^{-1}\left(\bar{\psi}(x)  \right)  }{h\left( \bar{f}^{-1}(\bar{\psi}(x))\right)}  \right)    \right)}{h\left( \bar{f}^{-1}(\bar{\psi}(x))\right)},
$$
where $\bar{f}(x) := h(x)\varphi(x)$ induces a bijection from $\psi(\Fn)$ to $\bar{\psi}(\Fn)$.
In particular, if $h(\psi(\Fn)) = \{c\}$ for some $c \in \F \setminus\{0\}$, then $f$ permutes $\Fn$ if and only if $\varphi$ permutes $\Fn$, in which case the compositional inverse of $f$
over $\Fn$ is given by
$$
f^{-1}(x) = c^{-1} \varphi^{-1}\left(x - g\left( c^{-1} \varphi^{-1}\left(\bar{\psi}(x)  \right)  \right) \right).
$$ 
 \end{cor}

 Note in the above that, in particular, when $h(x) = c \in \F \setminus \{0\}$, 
 the inverse of $f$ is given in terms of the inverse of $\varphi$ on $\Fn$, which can be obtained (see (1) above).
 As a result of Corollary \ref{zero} we also obtain Corollaries \ref{QT}, \ref{TQ}, and \ref{NQ}.
 Here is one example of Corollary \ref{zero} using the fact that $T(x)^q - T(x) = 0$.
 
 \begin{prop}\label{prop1}
 Let $G \in \mathbb{F}_{q^2}[x]$ be arbitrary, let $Q(x) := x^q - x$, let $g = T \circ G$, let $h(x) = c \in \F \setminus \{0\}$, and let $\varphi(x) = ax^q + bx$, where $a,b \in \F$ are such that $a \neq \pm b$. Then both
 $\varphi$ and
 $$
 f(x) = h(Q(x)) \varphi(x) + g(Q(x)) = c(ax^q + bx) + T \circ G \circ Q(x)
 $$
 are permutations of $\mathbb{F}_{q^2}$, and the compositional inverse of $f$ on $\mathbb{F}_{q^2}$ is given by
 $$
 f^{-1}(x) = \dfrac{ax^q - bx}{c\left(a^2 - b^2\right)} - \dfrac{T \circ G \circ \left( \dfrac{Q(x)}{c(b-a)}   \right)}{c(a+b)}.
 $$
 \end{prop}

 Another consequence of Theorem \ref{inv0} is the following.
 
 \begin{cor}[See Theorem 1, \cite{marcos}]\label{nice4}
 Let $\varphi \in \F[x]$ be a $q$-polynomial permuting $\Fn$, let $G \in \F[x]$, let $\gamma \in \Fn$, let $c = T(\gamma)$, and assume
 that
 $$
 f(x) = \varphi(x) + \gamma G(T(x))
 $$
 permutes $\Fn$. Then the inverse of $f$ on $\Fn$ is given by 
 $$
 f^{-1}(x) = \varphi^{-1} \left(x - \gamma G\left( \bar{f}^{-1}(T(x))\right)  \right),
 $$
 where $\bar{f}(x) := cG(x) + \varphi(1)x$ is a permutation of $\F$.
\end{cor}

Note in the case that $G$ is a $q$-polynomial, $\bar{f}|_{\F}(x) = (cG(1) + \varphi(1))x$; thus $\bar{f}^{-1}(T(x)) = T(x)/(cG(1) + \varphi(1))$
 and hence we can obtain the inverse of the permutation polynomial, $f$, by computing only $\varphi^{-1}$ (see (1) above). We thus obtain Corollary \ref{nice5}. 
We also note that if we take $G(x) = x^r$ and $c= \varphi(1)$, then $\bar{f}(x)$ in Corollary~\ref{nice4} has the form $c(x^r+x)$ whose inverse can be explicitly computed (see \cite{muratovic}, \cite{wang}, \cite{wang10}). 
Therefore the inverse of $f$ is again only dependent on $\varphi^{-1}$ which can also be obtained in terms of cofactors of the associate Dickson matrix as mentioned in (1). 
Another straightforward application of Corollary \ref{nice4} is the following example.

 \begin{prop}\label{prop2}
  As before, let $\varphi(x) = ax^q + bx$ where $a,b \in \F$ are such that $a \neq \pm b$, let $G(x) = x$, let $c = T(\gamma)$, where $\gamma \in \mathbb{F}_{q^2}$ is such that $a + b + c \neq 0$. Then both $\varphi$
  and 
  $$
  f(x) = \varphi(x) + \gamma T(x)
  $$
  are permutations of $\mathbb{F}_{q^2}$, and the compositional inverse of $f$ on $\mathbb{F}_{q^2}$ is given by
  $$
  f^{-1}(x) = \dfrac{ax^q - b x }{a^2 - b^2} - \dfrac{\left(a \gamma ^q - b \gamma  \right)T(x)   }{\left( a^2 - b^2 \right)(a + b + c)}.
  $$
 \end{prop}

 Another result of ours is the following, given in terms of the inverse of a linearized permutation polynomial, which can be obtained. 
 
 \begin{cor}[See Corollary 6.2, \cite{yuan}]\label{nice}
  Let $n$ and $k$ be positive integers such that $\gcd(n,k) = d > 1$, let $s$ be any positive integer with
  $s(q^k - 1) \equiv 0 \pmod{q^n - 1}$. Let
  $$
  L_1(x) = \sum_{i=0}^{n/d - 1} a_i x^{q^{id}}, \text{    } a_i \in \F,
  $$
  be a $q^d$-polynomial with $L_1(1) = 0$, let $L_2$ be a $q$-polynomial over $\F$,
  and let $G \in \Fn[x]$. Assume
  $$
  f(x) = \left(G(L_1(x))  \right)^s + L_2(x)
  $$
  permutes $\Fn$. Then $L_2$ is a permutation of $\Fn$, and the inverse of 
  $f$ on $\Fn$ is given by
  $$
  f^{-1}(x) = L_2^{-1}\left( x - \left( G \circ L_2^{-1} \circ L_1(x)  \right)^s   \right).
  $$
 \end{cor}

 As a result of Corollary \ref{nice} we obtain Corollary \ref{nice2} as well.
 One more example, in this case independent of Theorem \ref{inv0}, is the following. It 
 makes use of a slight modification of the permutation construction given in Theorem 6.1, \cite{akbary} (also appearing as Corollary 3.2 in \cite{yuan}),
 which now takes into account the fact that $p$-polynomials over $\F$ commute with $q$-polynomials over $\mathbb{F}_p$. 
 We define the natural map $v_{q,n} : \mathscr{L}_n(\Fn) \rightarrow \Fn^n$ 
 by $v_{q,n}(\sum_{i=0}^{n-1} a_i x^{q^i}) = (a_0 \ a_1 \cdots  a_{n-1})$ (for simplicity we write vectors horizontally).
 
 \begin{thm}[See Theorem 6.1, \cite{akbary}]\label{nice3}
 Let $q = p^m$ be the power of a prime number $p$, let $L_1,L_2$ be $p$-polynomials over $\F$, 
 let $L_3$ be a $q$-polynomial over $\mathbb{F}_p$, and let $w \in \Fn[x]$ such that $w(L_3(\Fn)) \subseteq \F$. Then
  $$
  f(x) = L_1(x) + L_2(x) w(L_3(x))
  $$
  permutes $\Fn$ if and only if $\bar{f}(x) := L_1(x) + L_2(x) w(x)$ is a permutation of $L_3(\Fn)$, 
  and $\varphi_y(x) := L_1(x) + L_2(x) w(y)$ is a permutation of $\ker(L_3)$ for any $y \in L_3(\Fn)$.
  If $f$ induces a permutation of $\Fn$, we have:
  
  (a) If
  $x \in \Fn$ is such that $\ker(\varphi_{\bar{f}^{-1}(L_3(x))}) \cap L_3(S_{L_3}) = \{0\}$, then the preimage of $x$ under $f$ is given by 
  $$
  f^{-1}(x) = \bar{f}^{-1}(L_3(x)) + \varphi^{-1}_{\bar{f}^{-1}(L_3(x))}|_{S_{L_3}}
  \left(x - L_3(x)  \right).
  $$
  
  (b) If $L_3$ is idempotent, then the compositional inverse of $f$ on $\Fn$ is given by
  $$
  f^{-1}(x) = \bar{f}^{-1}(L_3(x)) + \varphi^{-1}_{\bar{f}^{-1}(L_3(x))}|_{\ker(L_3)}\left(x - L_3(x)  \right),
  $$
  where, for any fixed $y \in L_3(\Fn)$, the coefficients, $v_{p,mn}\left(\varphi^{-1}_y|_{\ker(L_3)}\right)$, of a $p$-polynomial $\varphi^{-1}_y|_{\ker(L_3)}$ inducing the inverse map of $\varphi_y|_{\ker(L_3)}$, are a solution to the
  linear equation 
  $$
  v_{p,mn}\left(\varphi^{-1}_y|_{\ker(L_3)}\right) D_{\varphi_y} = v_{p,mn}(x - L_3(x)),
  $$
  where $D_{\varphi_y}$ is the $mn \times mn$ Dickson matrix associated with the $p$-polynomial $\varphi_y$.
  
  (c) If $x \in \Fn$ is such that $\varphi_{\bar{f}^{-1}(L_3(x))}$ permutes $L_3(\Fn)$, then $\varphi_{\bar{f}^{-1}(L_3(x))}$ permutes $\Fn$, and the preimage of $x$ under $f$ is given by
  $$
  f^{-1}(x) = \varphi^{-1}_{\bar{f}^{-1}(L_3(x))}(x).
  $$
  \end{thm}
  
  We remark that Theorem \ref{nice3} (b), (c), is a generalization, albeit non-explicit, of (2), (3), and its inverses (see \cite{coulter1}, \cite{wu}).  
  For example, letting $q$ be even and $n$ be odd, $a \in \F^*$, $L_3 = T$, $L_1(x) = a x^2$, $L_2(x) = x$, and $w(x) = L(x) + ax$, 
  we obtain $f(x) = ax^2 + x w(T(x)) = x(L(T(x)) + a T(x) + a x)$ in (3) with $L_3 = T$ being idempotent in this case of $q,n$.
  
  The rest of this paper goes as follows. In Section 2 we study a few basic and preliminary concepts which will be of use in further sections.
  In particular, we show in Theorem \ref{thm:LinearizedInverse} how to obtain linearized polynomials inducing inverse 
  maps on prescribed subspaces on which given linearized polynomials induce bijections.
  In Section 3 we prove Theorems \ref{inv0}, \ref{inv1}, where we give the compositional inverses
  of two classes of permutation polynomials given in 
  \cite{akbary}, \cite{yuan}, respectively, written in terms of the inverses of two polynomials permutating subspaces of $\Fn$. We then proceed to obtain several corollaries from them. 
  In Section 4 
  we explicitly obtain, in Theorem \ref{generalization} and Corollary \ref{cor:compositional_inverse}, the preimages and compositional inverse, respectively, 
  of a class of permutation
  polynomials generalizing that in (2) (see \cite{coulter1}) and (3) (see \cite{wu}) above. 
  Finally in Section 5 we prove Theorem \ref{thm: simple proof} below, which gives the explicit compositional inverse of a class of permutation
  polynomials generalizing that of a linearized permutation class whose inverse was recently obtained in \cite{wu_new}.
  See also Lemma \ref{lem: simple proof} which corresponds to the case when $G(x) = x$ and $c = 1$ below.
  In particular, the method employed here to obtain such a result,
  as an application of Theorem \ref{inv0}, seems considerably less complicated than that of \cite{wu_new}. 
  
  \begin{thm}\label{thm: simple proof}
  Let $\alpha \in \Fn$, $c \in \F^*$ and $G \in \Fn[x]$ be arbitrary. Then 
  $$
F(x) = c \left( x^q - x + T(\alpha x)  \right) +  G(T(\alpha x))^q - G(T(\alpha x))
$$
is a permutation polynomial over $\Fn$ if and only if $T(\alpha) \neq 0$ and the characteristic, $p$, of $\F$, does not divide $n$.
In this case the compositional inverse of $F$ is  
$$
F^{-1}(x) = c^{-1} \biggr[ T(\alpha)^{-1} n^{-1} \left(T(x) + B(x)   \right) +  T(\alpha)^{-1}T\left(\alpha G\left(c^{-1}n^{-1}T(x) \right) \right) -  G\left(c^{-1}n^{-1}T(x)\right)  \biggr],
$$
where the coefficients of $B(x) = \sum_{k=0}^{n-1}b_k x^{q^k} \in \Fn[x]$ are given by
 $$
 b_k = \sum_{j=1}^{n-1}j \alpha^{q^j} - n \sum_{l = k+1}^{n-1} \alpha^{q^l}, \ \ \ \ 0 \leq k \leq n-1.
 $$
\end{thm}

 
 \section{Preliminaries}
 
 In this section we deal with some basic preliminaries concepts which will be of use in further sections. 
 Linearized polynomials play a crucial role throughout the paper and we thus explore here some of its properties. 
 One can show that all the additive maps of $\Fn$ are induced by $p$-polynomials. 
 Clearly, $q$-polynomials are also $p$-polynomials.
 The set, $\Ln$, of all $q$-polynomials over $\Fn$, forms a unital associative non-commutative $\F$-algebra 
 with multiplication given by the composition of polynomials, and using the usual addition of polynomials and scalar multiplication by elements of $\F$. 
 In contrast, its subset of 
 all $q$-polynomials over $\F$ does form a commutative subalgebra. 
 Even more, $q$-polynomials over $\mathbb{F}_p$ commute with $p$-polynomials over $\F$.
 Note also that $\Ln$ is isomorphic to the algebras of the $n \times n$ matrices over $\F$, and the $n \times n$ Dickson matrices over $\Fn$.
 In particular, the subset of $q$-polynomial permutations over $\Fn$ 
 forms a group, under composition, that is isomorphic to the group of non-singular $n \times n$ matrices over $\F$.
 As stated in the Introduction, a linearized polynomial induces a permutation of $\Fn$ if and only if its associate Dickson matrix is non-singular.
 See the recent work in \cite{wu_lin} for known and new characterizations of $\Ln$.
 Note that a $q$-polynomial $L$ over $\F$ permutes $\F$ if and only if $L(1) \neq 0$, since $L(x) = L(1)x$ for all $x \in \F$.
 
 In the following sections we give expressions for the compositional inverses, or preimages in certain cases, 
 particularly written in terms of linearized polynomials inducing the inverse map over subspaces on which prescribed linearized polynomials induce a bijection. 
 Thus here in this section we show, in Theorem \ref{thm:LinearizedInverse}, how to obtain such linearized inverses over subspaces. 
See Lemma \ref{lem:LinearizedInverse} as well.
 To the knowledge of the authors this is an original result. Note that up to now it was known how to obtain the compositional inverse of a linearized {\em permutation} polynomial
 (see (1) in the Introduction and the following proposition). 
 
 \begin{prop}[{\bf Theorem 4.8, \cite{wu_lin}}]\label{LinInv}
  Let $L(x) = \sum_{i=0}^{n-1}a_i x^{q^i} \in \Ln$ be a linearized permutation polynomial and $D_L$ be its associate Dickson matrix. Assume $\bar{a}_i$ is the $(i,0)$-th cofactor of 
  $D_L$, $0 \leq i \leq n-1$. Then $\det(L) = \sum_{i=0}^{n-1} a^{q^i}_{n-i} \bar{a}_i$ and
  $$
  L^{-1}(x) = \dfrac{1}{\det(L)} \sum_{i=0}^{n-1}\bar{a}_i x^{q^i}.
  $$
 \end{prop}
 
 We start off by briefly noting a few basic properties of the set $S_\psi := \{ x - \psi(x) \mid x \in \Fn\}$, which will appear often throughout the paper.
 For $\psi \in \Ln$, we denote by $\psi^i$ the compositional power of $\psi$, that is, $\psi^i = \psi \circ \psi \circ \cdots \circ \psi$, $i$ times. 
 Denote by $\id \in \Ln$ the identity $\id(x) = x$.
 
\begin{lem}\label{0}
   Let $\psi \in \Ln$. Then the set $S_\psi := \{ x - \psi(x) \mid x \in \Fn\}$ is an $\F$-vector subspace of $\Fn$ and $\psi(S_\psi) \subseteq S_\psi$.
   Let $K \in \Ln$ such that $\ker(K) = S_\psi$. Then, for all $i \geq 0$, $K = K \circ \psi^i$ and $\ker(\psi^i) \subseteq S_\psi$. 
   If $\psi ^l = \psi^{l+1}$ for some positive integer $l$, then $S_\psi = \ker(\psi^l)$. 
  \end{lem}

  \begin{proof}
   The reader can check that $S_\psi$ is indeed a vector space over $\F$ and $\psi(S_\psi) \subseteq S_\psi$. Since $\ker(K) = \im(\id - \psi)$, then for all $x \in \Fn$ we have
   $K \circ (\id - \psi)(x) = 0$ and thus $K(x) = K (\psi(x)) = K(\psi(\psi(x))) = \cdots$, i.e., $K = K \circ \psi^i$ for each $i \geq 0$.
   If $x \in \ker(\psi^i)$, then 
   $K(x) = K(\psi^i(x)) = K(0) = 0$. Hence $\ker(\psi^i) \subseteq \ker(K) = S_\psi$.
   If $\psi^l = \psi^{l+1}$, then $S_\psi \subseteq \ker(\psi^l)$; it follows that $S_\psi = \ker(\psi^l)$. 
   \end{proof}
   
   In particular, if $\psi$ is idempotent, i.e., $\psi^2 = \psi$, then $S_\psi = \ker(\psi)$. Moreover, 
   if $\psi$ is nilpotent, i.e., there exists a positive integer $l$ such that $\psi^l = 0$, then $S_\psi = \ker(0) = \Fn$.

  Recall the natural map $v_{q,n} : \mathscr{L}_n(\Fn) \rightarrow \Fn^n$ defined by $v_{q,n}(\sum_{i=0}^{n-1} a_i x^{q^i}) = (a_0 \ a_1\ \cdots \ a_{n-1})$ (for simplicity we write vectors horizontally).
  Let $X = (x, x^q, \ldots, x^{q^{n-1}})$ and let $X^t$ denote the transpose of $X$. The following lemma shows that the composition of two linearized polynomials is 
  essentially induced via a Dickson matrix multiplication. 
  
 \begin{lem}\label{composition}
  Let $\varphi, \psi \in \Ln$, let $D_\psi$ be the Dickson matrix corresponding to $\psi$. Then $\varphi(\psi(x)) = v_{q,n}(\varphi)D_\psi X^t$.
 \end{lem}

\begin{proof}
Write $v_{q,n}(\varphi) = (a_0 \ a_1 \cdots a_{n-1})$, $v_{q,n}(\psi) = (b_0 \ b_1 \cdots b_{n-1})$, for some coefficients $a_i,b_j \in \Fn$. Then
\begin{align*}
\varphi(\psi(x)) &= \sum_{i=0}^{n-1}a_i \left( \sum_{j=0}^{n-1} b_j x^{q^j}  \right)^{q^i}
= \sum_{i=0}^{n-1}\sum_{j=0}^{n-1}a_i b_j^{q^i} x^{q^{i+j}}
= \sum_{i=0}^{n-1}\sum_{k=0}^{n-1}a_i b_{k-i}^{q^i} x^{q^{k}}\\
&= \sum_{k=0}^{n-1}\sum_{i=0}^{n-1} a_i b_{k-i}^{q^i} x^{q^{k}},
\end{align*}
where the subscripts are reduced modulo $n$. Since the $ik$-th entry of $D_\psi$ is given by $b_{k-i}^{q^i}$, the result follows. 
\end{proof}

We recall the well known concept of a projection operator on a vector space. These are the idempotent linear transformations. Let $U$ be a vector space which we write
as the internal direct sum of two subspaces $V, W$, that is, $U = V \oplus W$, where $U = V + W$ and $V \cap W = \{0\}$. We can write any element 
$u \in U$ uniquely as a sum $u = v + w$
where $v \in V$ and $w \in W$. Then the map $P : U \rightarrow U$ defined by $P(v + w) = w$ is called a {\em projection operator} of $U$ onto $W$ along $V$. 
It is easy to see that $P$ is a well-defined linear transformation with kernel $V$ and image $W$, and that $P(u) = u$ if and only if $u \in W$.
In fact, $P$ is a projection operator if and only if it is idempotent, i.e., $P(P(u)) = P(u)$ for all $u \in U$. 
Note that the existence of a projection operator onto any subspace is guaranteed, and thus we can always find 
at least one idempotent with a prescribed kernel. 

The following lemma gives the inverse map over subspaces on which a prescribed linear transformation is a bijection. 
It is true in general for arbitrary vector spaces and linear transformations.

\begin{lem}\label{lem:LinearizedInverse}
 Let $V, \bar{V}$, be two equally sized subspaces of a vector space $U$ over a field, 
 and let the linear operator $\varphi : U \rightarrow U$ be bijective from $V$ to $\bar{V}$. 
 There exists an idempotent $K : U \rightarrow U$ such that $\ker(K) = V$. Define the linear transformation $L : \varphi(U) \rightarrow V$ by $L(\varphi(u)) = u - K(u)$.
 Then $L|_{\bar{V}}$ is the inverse map of $\varphi|_V$.
 \end{lem}
 
 \begin{proof}
 Since we can construct a projection operator with kernel $V$, it is clear that such an idempotent $K$ exists. 
 First note that since $V = \ker(K) = \{u - K(u) \mid u \in U \}$, then $L(\varphi(U)) = V$. 
  To show that 
  $L$ is well-defined, assume that $\varphi(u) = \varphi(w)$ and $u \neq w$ (the case $u = w$ is trivial). We need to show that $L(\varphi(u)) = L(\varphi(w))$. 
  Because $\varphi$ is injective on $\ker(K)$, then $\ker(\varphi) \cap \ker(K) = \{0\}$; 
  thus, since $w - u \in \ker(\varphi)$, necessarily $w - u \notin \ker(K)$, otherwise $u = w$, a contradiction. The fact that $K$ is idempotent gives 
  $U = \ker(K) \oplus K(U)$, from which 
  it follows that $w - u \in K(U) \setminus \{0\}$, say $w - u = K(z)$ for some $z \in U \setminus \ker(K)$. 
  Then $\varphi(L(\varphi(u)) - L(\varphi(w))) = \varphi(u - w - K(u - w)) =  \varphi(K(w - u)) =  \varphi(K(K(z))) = \varphi(K(z)) = \varphi(w-u) = 0$; 
  hence $L(\varphi(u)) - L(\varphi(w)) \in \ker(\varphi)$. 
  Since $K$ is idempotent, $u - w - K(u - w) \in \ker(K)$. 
  Therefore $L(\varphi(u)) - L(\varphi(w)) \in \ker(\varphi) \cap \ker(K) = \{0\}$; thus $L(\varphi(u)) - L(\varphi(w)) = 0$ and so $L$ is well-defined.
  The reader can check that $L$ is a linear transformation.
  To see that $L|_{\bar{V}} = \varphi^{-1}|_{\bar{V}}$, note that for all $v \in V = \ker(K)$, $L(\varphi(v)) = v - K(v) = v$, 
  from which we get $\varphi(L(\varphi(v))) = \varphi(v)$. 
  Since $\varphi(V) = \bar{V}$, this 
  is the same as $\varphi(L(\bar{v})) = \bar{v}$ for all $\bar{v} \in \bar{V}$. Then $L|_{\bar{V}} = \varphi^{-1}|_{\bar{V}}$ by the uniqueness of inverse maps.
  \end{proof}
  
  We can now apply Lemma \ref{composition} and \ref{lem:LinearizedInverse} to obtain linearized polynomials inducing the inverse map 
  over subspaces on which a given linearized polynomial induces a bijection. 
  Obtaining the coefficients of such linearized polynomials is equivalent to solving a system of linear equations.

\begin{thm}\label{thm:LinearizedInverse}
   Let $V, \bar{V}$, be two equally sized $\F$-subspaces of $\Fn$, let $\varphi \in \Ln$ induce a bijection from $V$ to $\bar{V}$, and let $D_\varphi$
   be the associate Dickson matrix of $\varphi$. There exists $K \in \Ln$ inducing an idempotent endomorphism of $\Fn$ such that $\ker(K) = V$. 
   Then the coefficients, $\bar{c} = (c_0 \ c_1 \cdots c_{n-1})$, of one of the linearized
   polynomials inducing the inverse map of $\varphi|_V$, satisfy the linear equation $\bar{c} D_\varphi = v_{q,n}(\id - K)$.
  \end{thm}
  
  \begin{proof}
   We make use of the fact that when $U = \Fn$ in Lemma \ref{lem:LinearizedInverse}, all the linear transformations are induced by linearized polynomials.
   Now define the map $\bar{L} : \Fn \rightarrow \Fn$ by $\bar{L} = L \circ P$, where $L$ is given in Lemma \ref{lem:LinearizedInverse} (with $U = \Fn$), and $P$ is a projection operator of $\Fn$ onto $\varphi(\Fn)$.
   Note that $\bar{L}$ is a well-defined linear operator. The fact that $P$ is a projection operator onto $\varphi(\Fn)$ gives $\bar{L}(\varphi(x)) = L \circ P (\varphi(x)) = L(\varphi(x)) = x - K(x)$.
   In particular, $\bar{L}|_{\bar{V}} = L|_{\bar{V}} = \varphi^{-1}|_{\bar{V}}$. Thus we have obtained a linear operator of $\Fn$ inducing the inverse map of $\varphi|_V$.
   But all linear operations of $\Fn$ are induced by linearized polynomials. Since, additionally, $V,\bar{V}$, are $\F$-subspaces and $\varphi \in \Ln$, it follows that
   the map of $\bar{L}$ can be induced by a linearized polynomial in $\Ln$, say $R$, inducing the inverse
   map of $\varphi|_V$, and necessarily satisfying the equation $R(\varphi(x)) = x - K(x)$.
   Then using Lemma \ref{composition} we obtain the linear equation $v_{q,n}(R)D_\varphi  = v_{q,n}(\id - K)$ 
   (after cancelling out the arbitrary $X^t$ on both sides) as required. 
  \end{proof}
  
  Note that if $\varphi$ induces a permutation of $\Fn$, i.e., $D_\varphi$ is non-singular, then the required idempotent is $K = 0$, and 
  the coefficients of the compositional inverse of $\varphi$ are given by $\bar{c} = v_{q,n}(\id)D_{\varphi}^{-1}$, i.e., the coefficients in the first row of $D_\varphi^{-1}$. See (1) in the Introduction 
  and Proposition \ref{LinInv} for details.
  
  \begin{rmk}
   In the case that the coefficients of $\varphi \in \Ln$ belong to $\F$, $D_\varphi$ becomes a circulant matrix. 
   Then if $V$ is an $\F$-subspace and if the characteristic $p$ of $\Fn$ does not divide $n$, we can quickly solve the linear equation
   $\bar{c} D_{\varphi} = v_{q,n}(\id - K)$ by using the Circular Convolution Theorem for DFTs together with an FFT. See \cite{jones} for details on circulants and DFTs.
  \end{rmk}

   An immediate consequence is the following.
  
  \begin{cor}\label{cor:IdempotentLinearEqn}
   Let $\varphi, \psi \in \Ln$ such that $\psi$ induces an idempotent endomorphism of $\Fn$ and $\varphi$ induces a bijection from $S_\psi = \ker(\psi)$ 
   to an $\F$-subspace $\bar{V}$ of $\Fn$. 
   Then the coefficients, $\bar{c} = (c_0 \ c_1 \cdots c_{n-1})$, of a linearized polynomial inducing the inverse map from $\bar{V}$ to $S_\psi$, satisfy the linear equation
   $\bar{c} D_\varphi = v_{q,n}(\id - \psi)$.
   \end{cor}
   
   \begin{eg}\label{eg:IdempotentLinearEqn}
    Assume that the characteristic $p$ does not divide $n$ and work over $\Fn$, where $q = p^m$ as before. We have $n^{-1}T(n^{-1}T(x)) = n^{-2}T^2(x) = n^{-2}nT(x) = n^{-1}T(x)$ and so 
    $n^{-1}T$ is idempotent with kernel $\ker(T) = \{ \beta^q - \beta \mid \beta \in \Fn \}$. 
    Let $\varphi(x) := x^p + cx \in \mathscr{L}_{mn}(\mathbb{F}_{p^{mn}})$, where $c \in \F^*$ satisfies $c^{(q-1)/(p-1)} = (-1)^m$,
    and let $D_\varphi$ be its associate $mn \times mn$ Dickson matrix given by
    $$
   D_{\varphi} = \left( \begin{array}{cccccc}
                 c & 1 & 0 & 0 &\cdots & 0\\
                 0 & c^p & 1 & 0 & \cdots & 0\\
                 \vdots & \vdots & & & & \vdots\\
                1 & 0 & 0 & 0 & \cdots & c^{p^{mn-1}}
                \end{array} \right).
   $$
    Note that $\varphi(\ker(T)) \subseteq \ker(T)$. If $\varphi$ permutes the $\mathbb{F}_p$-subspace, $\ker(T)$, of $\Fn$, then by Theorem \ref{thm:LinearizedInverse}, the coefficients, 
    $\bar{d} = (d_0 \ d_1 \cdots d_{mn-1})$, of one of the linearized polynomials inducing the inverse permutation of $\varphi|_{\ker(T)}$, are a solution 
    to the linear equation 
    \begin{align*}
    \bar{d} D_\varphi &= v_{p,mn}\left(\id - n^{-1} T\right) = v_{p,mn}\left(x - n^{-1}\sum_{k=0}^{n-1}x^{p^{km}}\right) \\
    &= -n^{-1}(1-n,0,0, \ldots, 0, 1, 0,0, \ldots, 0, 1, 0,0, \ldots, 0),
    \end{align*}
    where the $1$ entries occur at indices (which start at $0$ and end at $mn-1$) that are non-zero multiples of $m$. Solving this system
    we obtain a solution 
    $$
    d_{km+j} = n^{-1}(-1)^j c^{-(p^{j+1} - 1)/(p-1)}(n-1-k),
    $$
    where $0 \leq k \leq n-1$ and $0 \leq j \leq m-1$. 
    Thus we have found a linearized polynomial inducing the inverse map of $\varphi|_{\ker(T)}$. In particular, $\varphi$ induces a permutation of $\ker(T)$ as maps are invertible
    if and only if they are bijections.
    \end{eg}

  \begin{rmk}\label{rmk:num_idempotent}
  There exist several linearized polynomials inducing idempotents operations of $\Fn$.
  In fact, the number of linearized polynomials in $\Ln$ inducing such idempotent maps is given by
  $$
  \sum_{k=0}^{n}\dfrac{|\GL(n,q)|}{|\GL(k,q)| |\GL(n-k,q)|},
  $$
  where $\GL(k,q)$ denotes the group of non-singular $k \times k$ matrices over $\F$ with (well-known) order given by $|\GL(k,q)| = \prod_{i=0}^{k-1}(q^k - q^i)$ for $k \geq 1$,
  and $|\GL(0,q)| = 1$ by convention.
  Indeed, we know that internal direct-sum decompositions of $\Fn$, into ordered 
  pairs of trivially-intersected $\F$-subspaces, are in correspondence with idempotent $\F$-linear operators of $\Fn$, 
  which in turn correspond to elements in $\Ln$ inducing such idempotent maps.
  Fixing an integer $k$, where $0\leq k \leq n$, we note that the bases of $\Fn$ over $\F$ are in bijective correspondence with ordered pairs of disjoint bases with dimensions
  $k$ and $n-k$, respectively, over $\F$. But to each $\F$-subspace of $\Fn$ of dimension, say $l$, corresponds $|\GL(l,q)|$ bases of it. 
  Then there are $\frac{|\GL(n,q)|}{|\GL(k,q)||\GL(n-k,q)|}$ decompositions of $\Fn$ into ordered pairs of trivially-intersected $\F$-subspaces with dimensions $k$ and $n-k$, 
  respectively. The claim now follows. For example, when $n=2$ and $q = 2,3,4,5$, this number is $8,14,22,32$, respectively.
 \end{rmk}

We now place our attention to other preliminary lemmata which will be of use in the following sections.
The following lemma is a generalization of Lemma 3.1 in \cite{wu} and thus serves in tackling more general classes of permutations of arbitrary finite fields.

\begin{lem}\label{i0}
 Let $q = p^m$ be a prime power, let $\psi, \bar{\psi} \in \Fn[x]$ be additive, and let $f \in \Fn[x]$. 
 Define the map $\phi_\psi : \Fn \rightarrow \psi(\Fn) \oplus S_\psi$ 
 by $\phi_\psi(x) = (\psi(x), x - \psi(x))$. Then $\phi_\psi$ is injective for any such $\psi$, and, for any $(y,z) \in \phi_\psi(\Fn)$, 
 $\phi_\psi^{-1}(y,z) = y + z$. Then $f$ permutes $\Fn$ if and only if the map 
 $F : \phi_\psi(\Fn) \rightarrow \phi_{\bar{\psi}}(\Fn)$, given by $F = \phi_{\bar{\psi}} \circ f \circ \phi_\psi^{-1}$, is bijective.
 In this case, $F^{-1} = \phi_\psi \circ f^{-1} \circ \phi_{\bar{\psi}}^{-1}$.
 \end{lem}
 
 \begin{proof}
  It is easy to check that $\phi_\psi$ is well-defined and injective for any such $\psi$. Then $|\phi_\psi(\Fn)| =  |\phi_{\bar{\psi}}(\Fn)| = |\Fn|$.
  Hence, since $F \circ \phi_\psi = \phi_{\bar{\psi}} \circ f$, it follows that $f$ is 
  a permutation if and only if $F$ is a bijection (injectivity implies bijectivity since the respective 
  domain and codomain of $f, F$, are equally sized and finite). 
 \end{proof}

 This lemma proves valuable in that it provides a way of computing the inverses of certain classes of polynomials by computing the inverses
 of two other polynomials, one of these linearized, on subspaces; this should be easier. 
 
 First it would be interest to determine when is a linear map $\varphi$ a bijection between two similar subspaces $S_\psi, S_{\bar{\psi}}$. 
 Under certain restrictions on $\varphi,\psi,\bar{\psi}$, Lemma \ref{i2} shows when this happens. 
 Now we make use of the following lemma which is a slight (more general) variation of the AGW Criterion given in Lemma 1.2, \cite{akbary}.
 
 \begin{lem}\label{i1}
  Let $A,\bar{A},S,\bar{S}$, be finite sets with $|A| = |\bar{A}|$ and $|S| = |\bar{S}|$, and let 
  $f : A \rightarrow \bar{A}$, $\bar{f} : S \rightarrow \bar{S}$, $\lambda : A \rightarrow S$, and 
  $\bar{\lambda} : \bar{A} \rightarrow \bar{S}$ be maps such that $\bar{\lambda} \circ f = \bar{f} \circ \lambda$.
  If both $\lambda$ and $\bar{\lambda}$ are surjective, then the following two statements are equivalent:
  \\
  (i) $f$ is a bijection from $A$ to $\bar{A}$; and
  \\
  (ii) $\bar{f}$ is a bijection from $S$ to $\bar{S}$ and $f$ is injective on $\lambda^{-1}(s)$ for each $s \in S$.
  
  Furthermore, if $A, \bar{A}, S, \bar{S},$ form groups under some operation, $+$ (although the groups may not be abelian), and $f,\bar{f}, \lambda, \bar{\lambda},$ are homomorphisms on the respective groups, then $f$
  is injective on $\lambda^{-1}(s)$ for each $s \in S$ if and only if $\ker(f) \cap \ker(\lambda) = \{0\}$.
 \end{lem}
 
 \begin{proof}
The proof of (i), (ii), is identical to the proof of Lemma 1.2 in \cite{akbary} and we omit it. Now for the group case, assume 
$\alpha,\beta \in \lambda^{-1}(s) := \{ a \in A  \mid \lambda(a) = s \}$ such that $\alpha \neq \beta$. Then $0 \neq f(\alpha) - f(\beta) = f(\alpha - \beta)$ if and only if $\ker(f) \cap \ker(\lambda) = \{0\}$ (since $\alpha - \beta \in \ker(\lambda)$) as required.
\end{proof}

The following easy consequence gives us a criteria for determining when does a linearized polynomial induce a permutation of the finite field in question. 
We will often apply this useful result in the following sections.

\begin{lem}\label{corlem}
 Let $\varphi,\psi, \bar{\psi} \in \Fn[x]$ be additive such that $\varphi \circ \psi = \bar{\psi} \circ \varphi$ and $|\psi(\Fn)| = |\bar{\psi}(\Fn)|$. Then $\varphi$
 induces a permutation of $\Fn$ if and only if $\ker(\varphi) \cap \ker(\psi) = \ker(\varphi) \cap \psi(\Fn) = \{0\}$.
\end{lem}

\begin{lem}\label{i2}
  Let $\varphi,\psi, \bar{\psi} \in \Fn[x]$ be additive such that $\varphi \circ \psi = \bar{\psi} \circ \varphi$, $|\psi(\Fn)| = |\bar{\psi}(\Fn)|$, and $|S_\psi| = |S_{\bar{\psi}}|$.
  Then $\varphi$ induces a bijection between $S_\psi$ and $S_{\bar{\psi}}$ if and only if $\ker (\varphi) \cap \ker(\psi) = \ker(\varphi) \cap \psi(S_\psi) = \{0\}$.
 \end{lem}

 \begin{proof}
  First note that because $|\psi(\Fn)| = |\bar{\psi}(\Fn)|$, then $|\ker(\psi)| = |\ker(\bar{\psi})|$. As
  $\ker(\psi) \subseteq S_\psi$, $\ker(\bar{\psi}) \subseteq S_{\bar{\psi}}$, and $|S_\psi| = |S_{\bar{\psi}}|$, 
  it follows that $|\psi(S_\psi)| = |\bar{\psi}(S_{\bar{\psi}})|$. Moreover, since $\psi(S_\psi) \subseteq S_\psi$ and $\varphi \circ \psi = \bar{\psi} \circ \varphi$,
  we deduce that $\varphi(S_\psi) \subseteq S_{\bar{\psi}}$ and $\varphi(\psi(S_\psi)) \subseteq \bar{\psi}(S_{\bar{\psi}})$. 
  Now Lemma \ref{i1} finishes the proof if we let $A = S_\psi$, $\bar{A} = S_{\bar{\psi}}$, $S = \psi(S_\psi)$, 
  $\bar{S} = \bar{\psi}(S_{\bar{\psi}})$, $f = \bar{f} = \varphi$, $\lambda = \psi$, and $\bar{\lambda} = \bar{\psi}$.
 \end{proof}

\section{Inverses of some classes of permutations}

In this section we apply the method highlighted in the Introduction and proceed to obtain inverses of several classes of permutations.
The list of chosen permutations for which we study their inverses is by no means exhaustive, although the method could potentially be used to 
obtain several more inverses of other classes of permutations not given here. Moreover, more corollaries than we have given here may be possible
from our results, which could be obtained by choosing specific instances of the variables $\varphi, \psi$, $g$, $h$, etc..

\begin{proof}[{\bf Proof of Theorem \ref{inv0}}]
 First note that since $|S_\psi| = |S_{\bar{\psi}}|$ and $\ker(\varphi) \cap \psi(S_\psi) = \ker(\varphi) \cap \ker(\psi) = \{0\}$, then
 $\varphi$ is a bijection from $S_\psi$ to $S_{\bar{\psi}}$ (and hence invertible) by Lemma \ref{i2}.
 Now letting $y := \psi(x)$, $z := x - \psi(x)$, 
  we obtain
 \begin{align*}
  Y &:= \bar{\psi}(f(x)) = h(\psi(x)) \varphi(\psi(x)) + \bar{\psi}(g(\psi(x))) = h(y)\varphi(y) + \bar{\psi}(g(y)) \\
  &= \bar{f}(y) \in \bar{\psi}(\Fn); \\
  Z &:= f(x) - \bar{\psi}(f(x)) = h(y) \varphi(x) + g(y) - h(y) \varphi(y) - \bar{\psi}(g(y)) \\
  &= h(y) \varphi(z) + g(y) - \bar{\psi}(g(y)) \in S_{\bar{\psi}}.
  \end{align*}
  Then by the definition of $F$ from Lemma \ref{i0}, we get
  \begin{equation*}
  \begin{split}
  F(\phi_\psi(x)) &= F(y,z) = \phi_{\bar{\psi}}\left(f \left( \phi_\psi^{-1}(y,z) \right) \right) = \phi_{\bar{\psi}}(f(y+z))\\
  &= \phi_{\bar{\psi}}(f(x)) = (\bar{\psi}(f(x)), f(x) - \bar{\psi}(f(x)))\\
  &= (Y,Z).
  \end{split}
  \end{equation*}
  Hence $F$ is defined by 
  $$
  F(y,z) = (\bar{f}(y), h(y) \varphi(z) + g(y) - \bar{\psi}(g(y))) = (Y,Z)
  $$
  for $y \in \psi(\Fn)$ and $z \in S_\psi$. Now, from Lemma \ref{i0} we know that
  $f$ is a permutation of $\Fn$ if and only if $F$ induces a bijection between $\phi_{\psi}(\Fn)$ and $\phi_{\bar{\psi}}(\Fn)$. Since $\bar{f}$ induces a bijection between $\psi(\Fn)$, 
  $\bar{\psi}(\Fn)$, and in our case $h(y) \varphi(z) + g(y) - \bar{\psi}(g(y))$ induces a bijection between $S_{\psi}$, $S_{\bar{\psi}}$, for any $y \in \psi(\Fn)$,
  we can obtain the inverse of $F$ from the inverses of $\bar{f}|_{\psi(\Fn)}$ and $\varphi|_{S_\psi}$.
  Now since $y = \bar{f}^{-1}(Y)$, we obtain
  \begin{align*}
  z &= \varphi^{-1}|_{S_{\bar{\psi}}}    \left( \dfrac{Z - g(y) +  \bar{\psi}(g(y))  }{h(y)}  \right) \\
  &= \varphi^{-1}|_{S_{\bar{\psi}}}   \left( \dfrac{Z - g(\bar{f}^{-1}(Y)) +  \bar{\psi}(g(\bar{f}^{-1}(Y)))  }  {h(\bar{f}^{-1}(Y))}  \right).
  \end{align*}
  Thus,
  $$
  F^{-1}(Y,Z) = \left(\bar{f}^{-1}(Y), \varphi^{-1}|_{S_{\bar{\psi}}}   \left( \dfrac{Z - g(\bar{f}^{-1}(Y)) + \bar{\psi}(g(\bar{f}^{-1}(Y)))  }  {h(\bar{f}^{-1}(Y))}  \right) \right).
  $$
  As $f^{-1} = \phi^{-1}_{\psi} \circ F^{-1} \circ \phi_{\bar{\psi}}$ by Lemma \ref{i0}, then letting $Y = \bar{\psi}(x)$ and $Z = x - \bar{\psi}(x)$, we finally obtain
  $$
  f^{-1}(x) = \bar{f}^{-1}\left(\bar{\psi}(x)\right) + \varphi^{-1}|_{S_{\bar{\psi}}}   \left( \dfrac{x - \bar{\psi}(x) - g\left(\bar{f}^{-1}\left(\bar{\psi}(x)\right)\right) + \bar{\psi}\left(g\left(\bar{f}^{-1}\left(\bar{\psi}(x)\right)\right)\right) }     {h\left(\bar{f}^{-1}\left(\bar{\psi}(x)\right)\right)}   \right),
  $$
  as required. 
  
  If $\varphi$ is a bijection from $\psi(\Fn)$ to $\bar{\psi}(\Fn)$, then by Lemma \ref{corlem} it permutes $\Fn$ (since additionally $\ker(\varphi) \cap \ker(\psi) = \{0\}$
  because $f$ is a permutation). We claim that for any $\bar{y} \in \bar{\psi}(\Fn)$, $\bar{f}^{-1}(\bar{y})$ satisfies
  $$
  \bar{f}^{-1}(\bar{y}) = \varphi^{-1}\left( \dfrac{\bar{y} - \bar{\psi}\left( g\left( \bar{f}^{-1}(\bar{y}) \right) \right)}{h\left( \bar{f}^{-1}(\bar{y}) \right)} \right).
  $$
  Indeed,
  \begin{equation*}
   \begin{split}
    \bar{f}\left(\bar{f}^{-1}(\bar{y})  \right) &= h\left( \bar{f}^{-1}(\bar{y})\right) \varphi\left( \bar{f}^{-1}(\bar{y})\right) + \bar{\psi}\left( g\left( \bar{f}^{-1}(\bar{y})\right)  \right)\\
   &= h\left( \bar{f}^{-1}(\bar{y})\right) \varphi \left(  \varphi^{-1}\left( \dfrac{\bar{y} - \bar{\psi}\left( g\left( \bar{f}^{-1}(\bar{y}) \right) \right)}{h\left( \bar{f}^{-1}(\bar{y}) \right)} \right) \right) + \bar{\psi}\left( g\left( \bar{f}^{-1}(\bar{y})\right)  \right)\\
   &= \bar{y}.
   \end{split}
   \end{equation*}
Similarly one can show that $\bar{f}^{-1}\left(\bar{f}(y)  \right) = y$ for each $y \in \psi(\Fn)$.
  Now the fact that $\varphi^{-1}$ must be additive yields the result.
  \end{proof}
  
  \begin{cor}[See Theorem 5.5, \cite{akbary}]
 Let $a \in \F$, let $b \in \Fn$, let $P, L \in \F[x]$ be $q$-polynomials. Let $H \in \Fn[x]$ such that $H(L(\Fn)) \subseteq \F \setminus \{-a\}$, and assume
 $$
 f(x) = a P(x) + (P(x) + b) H(L(x))
 $$
 permutes $\Fn$. Then $\ker(P) \cap \ker(L) = \{0\}$, and if additionally $\ker(P) \cap L(S_{L}) = \{0\}$, the inverse of $f$ on $\Fn$ is given by
 $$
 f^{-1}(x) = \bar{f}^{-1}(L(x)) + \dfrac{P^{-1}|_{S_L} \left(x - L(x) - b H\left( \bar{f}^{-1}(L(x))   \right)  + L\left( b H \left( \bar{f}^{-1}(L(x))   \right)      \right)      \right)} {a + H\left(\bar{f}^{-1}(L(x))   \right)},
 $$
 where $\bar{f}(x) = (a + H(x))P(x) + L(b) H(x)$ is a permutation of $L(\Fn)$.
\end{cor}

\begin{proof}
This is Theorem \ref{thm0} and \ref{inv0} with $h = a + H$, $\varphi = P$, $\psi = L$, and $g = bH$.
\end{proof}

\begin{proof}[{\bf Proof of Corollary \ref{zero}}]
  By Theorem \ref{thm0}, if $f$ is a permutation of $\Fn$, 
  then $\bar{f}(x) = \bar{\psi}(g(x)) + h(x) \varphi(x) = h(x) \varphi(x)$ is a bijection from $\psi(\Fn)$ to $\bar{\psi}(\Fn)$. 
  In particular, $\bar{f}|_{\psi(\Fn)}$ is surjective. Thus for any $y \in \Fn$, there exists an $x \in \Fn$ such that 
  $\bar{\psi}(y) = h(\psi(x)) \varphi(\psi(x)) = \varphi(\psi(x) h(\psi(x)) ) = \varphi( \psi(x h(\psi(x))))$. Hence $\varphi$ is a surjection from $\psi(\Fn)$ to $\bar{\psi}(\Fn)$.
  But $|\psi(\Fn)| = |\bar{\psi}(\Fn)|$. Then $\varphi|_{\psi(\Fn)}$ is a bijection from $\psi(\Fn)$ to $\bar{\psi}(\Fn)$. In particular, if $h(\psi(\Fn)) = c \in \F^*$, then 
  $\bar{f}$ is a bijection from $\psi(\Fn)$ to $\bar{\psi}(\Fn)$ if and only if so is $\varphi$. Now, since $\ker(\varphi) \cap \ker(\psi) = \{0\}$ additionally (because $f$ is a permutation), 
  $\varphi$ is a permutation of $\Fn$ by Lemma \ref{corlem}. Now the rest follows from Theorem \ref{inv0} together with the fact that for any $\bar{y} \in \bar{\psi}(\Fn)$, 
  $\bar{f}^{-1}(\bar{y}) = \varphi^{-1}(\bar{y})/h(\bar{f}^{-1}(\bar{y}))$. 
 \end{proof}
 
We denote $Q(x) := x^q - x$, and $N(x) := x^{(q^n-1)/(q-1)}$. Since $q$-polynomials over $\F$ commute under composition, and
$T \circ Q = Q \circ T = Q \circ N = 0$, the following corollaries are straightforward applications of Corollary \ref{zero}.
 
 \begin{cor}[See Theorem 5.8, \cite{akbary}]\label{QT}
 Let $\varphi \in \F[x]$ be a $q$-polynomial, let $G \in \Fn[x]$, let $h \in \Fn[x]$ such that $h(\F) \subseteq \F \setminus \{0\}$.
 Assume that
 $$
 f(x) = h(T(x)) \varphi(x) + Q(G(T(x)))
 $$
 is a permutation of $\Fn$. Then $\varphi$ is a permutation of $\Fn$ and the inverse of $f$ on $\Fn$ is given by 
 $$
 f^{-1}(x) = \dfrac{\varphi^{-1}\left(x - Q \circ G \left( \dfrac{\varphi^{-1}\left(T(x)  \right)  }{h\left( \bar{f}^{-1}(T(x))\right)}  \right)    \right) }{h\left( \bar{f}^{-1}(T(x))\right)},
 $$
 where $\bar{f}(x) := \varphi(1) x h(x)$ is a permutation of $\F$.
\end{cor}

 \begin{cor}[See Theorem 5.10, \cite{akbary}]\label{TQ}
  Let $\varphi \in \F[x]$ be a $q$-polynomial, let $G \in \Fn[x]$, let $h \in \Fn[x]$ such that $h(Q(\Fn)) \subseteq \F \setminus \{0\}$.
 Assume that
 $$
 f(x) = h(Q(x)) \varphi(x) + T(G(Q(x)))
 $$
 is a permutation of $\Fn$. Then $\varphi$ is a permutation of $\Fn$ and the inverse of $f$ on $\Fn$ is given by 
 $$
 f^{-1}(x) = \dfrac{\varphi^{-1}\left(x - T \circ G \left( \dfrac{\varphi^{-1}\left(Q(x)  \right)  }{h\left( \bar{f}^{-1}(Q(x))\right)}  \right)    \right)}{h\left( \bar{f}^{-1}(Q(x))\right)},
 $$
 where $\bar{f}(x) := h(x) \varphi(x)$ is a permutation of $Q(\Fn)$.
 \end{cor}
 
  \begin{proof}[{\bf Proof of Proposition \ref{prop1}}]
  Follows directly from Corollary \ref{TQ} by noticing that $\varphi^{-1}(x) = (a x^q - b x)/(a^2 - b^2)$ on $\mathbb{F}_{q^2}$, $\varphi^{-1}(Q(x)) = Q(x)/(b-a)$, and $T(x)^q = T(x)$ for any $x \in \mathbb{F}_{q^2}$.
 \end{proof}

\begin{cor}[See Theorem 5.10, \cite{akbary}]\label{NQ}
  Let $\varphi \in \F[x]$ be a $q$-polynomial, let $G \in \Fn[x]$, let $h \in \Fn[x]$ such that $h(Q(\Fn)) \subseteq \F \setminus \{0\}$.
  Let $N(x) := x^{(q^{n} - 1)/(q-1)}$, and assume that
 $$
 f(x) = h(Q(x)) \varphi(x) + N(G(Q(x)))
 $$
 is a permutation of $\Fn$. Then $\varphi$ is a permutation of $\Fn$ and the inverse of $f$ on $\Fn$ is given by 
$$
 f^{-1}(x) = \dfrac{\varphi^{-1}\left(x - N \circ G \left( \dfrac{\varphi^{-1}\left(Q(x)  \right)  }{h\left( \bar{f}^{-1}(Q(x))\right)}  \right)    \right)}{h\left( \bar{f}^{-1}(Q(x))\right)},
 $$
 where $\bar{f}(x) := h(x) \varphi(x)$ is a permutation of $Q(\Fn)$.
 \end{cor}
 
 \begin{proof}[{\bf Proof of Corollary \ref{nice4}}]
 In this case $h = 1$, $g = \gamma G$, and $\psi = T$.
\end{proof}

Letting $G(x) = x$ in Corollary \ref{nice4}, the following is straightforward.

\begin{cor}[See Corollary 3.5, \cite{yuan}]\label{nice5}
  Let $\varphi \in \F[x]$ be a $q$-polynomial permuting $\Fn$, let $\gamma \in \Fn$, let $c = T(\gamma)$, and assume
 that
 $$
 f(x) = \varphi(x) + \gamma T(x)
 $$
 permutes $\Fn$. Then $c + \varphi(1) \neq 0$ and the inverse of $f$ on $\Fn$ is given by 
 $$
 f^{-1}(x) = \varphi^{-1} \left(x - \dfrac{\gamma T(x)}{c + \varphi(1)} \right).
 $$
 \end{cor}
 

 Some other consequences of Theorem \ref{inv0} are the following corollaries. 
 
 \begin{thm}[{\bf Theorem 5.12 (a), \cite{akbary}}]\label{5.12}
  Let $k \geq 2$ be an even integer, let $L(x) = a x^q + b x$, where $a,b \in \F$ such that $a \neq \pm b$. Then both $L$ and 
  $$
  f(x) = L(x) + Q(x)^k
  $$
  are permutations of $\mathbb{F}_{q^2}$.
 \end{thm}

 \begin{cor}
  Using the same notations as in Theorem \ref{5.12}, the inverse of $f$ on $\mathbb{F}_{q^2}$ is given by
  $$
  f^{-1}(x) = \dfrac{ax^q - bx}{a^2 - b^2} - \dfrac{1}{a+b}\left(\dfrac{Q(x)}{a - b} \right)^k.
  $$
 \end{cor}
 
 \begin{proof}
  Here $\psi = Q$, $g(x) = x^k$, $\varphi = L$, $h = 1$, and $\bar{f}(x) = Q(x^k) + L(x)$ is a permutation of $Q(\Fn)$.
  The fact that $k \geq 2$ is even gives $Q[ (Q(x))^k] = 0$ for any $x \in \mathbb{F}_{q^2}$. Then $\bar{f}(Q(x)) = L(Q(x)) = Q(L(x))$; thus
  $\bar{f}^{-1}(Q(x)) = L^{-1}(Q(x))$. Then by Theorem \ref{inv0} we obtain
  $$
  f^{-1}(x) = L^{-1} \left( x - \left(  L^{-1}(Q(x))    \right)^k  \right).
  $$
  Note that since $a,b \in \F$ and $a \neq \pm b$, then $a^2 - b^2 \in \F^*$. The reader can check that $L^{-1}(x) = (ax^q - bx)/(a^2 - b^2)$ from which it follows that
  $L^{-1}(Q(x)) = Q(x)/(b - a)$. Then using the fact that $k$ is even we obtain
  \begin{equation*}
  \begin{split}
  L^{-1}\left( \left(L^{-1}(Q(x))  \right)^k  \right) &= L^{-1}\left( \left( \dfrac{Q(x)}{b - a}  \right)^k \right) = \dfrac{a\left(\dfrac{Q(x)}{b - a} \right)^{qk}- b \left(\dfrac{Q(x)}{b - a} \right)^{k} }{a^2 - b^2}\\
  &= \dfrac{a\left(\dfrac{Q(x)}{b - a} \right)^{k} - b \left(\dfrac{Q(x)}{b - a} \right)^{k} }{a^2 - b^2} \\
  &= \dfrac{1}{a+b}\left(\dfrac{Q(x)}{a - b} \right)^k,
  \end{split}
  \end{equation*}
  from which the result follows since $L^{-1}$ is additive.
 \end{proof}

 \begin{proof}[{\bf Proof of Corollary \ref{nice}}]
  We have $\psi = L_1$, $\varphi = L_2$ is a permutation of $\Fn$ and particularly of $L_1(\Fn)$ by Corollary 6.2 in \cite{ding}, $h = 1$, $g(x) = (G(x))^s$. From the proof 
  of Corollary 6.2 in \cite{yuan}, $L_1[(G(L_1(x)))^s] = 0$, and 
  $\bar{f} = L_2$ is a permutation of $L_1(\Fn)$. Now Theorem \ref{inv0} gives the result.
 \end{proof}

 The following is a direct consequence of the above and Corollary 6.2 in \cite{yuan}.
 
 \begin{cor}[See Corollary 6.3, \cite{yuan}]\label{nice2}
  Let $n$ and $k$ be positive integers such that $\gcd(n,k) = d > 1$, and let $s$ 
  be a positive integer with $s(q^k - 1) \equiv 0 \pmod{q^n - 1}$. Then
  $$
  f(x) = x + \left(x^{q^k} - x + \delta   \right)^s
  $$
  permutes $\Fn$ for any $\delta \in \Fn$, and the inverse of $f$ on $\Fn$ is given by
  $$
  f^{-1}(x) = x - \left(x^{q^k} - x + \delta   \right)^s.
  $$
 \end{cor}
 
 \begin{proof}
  This is Corollary \ref{nice} with $L_1(x) = x^{q^k} - x$, $G(x) = x + \delta$, $L_2(x) = x$.
  Since $L_2^{-1}(x) = x$, the result follows.
   \end{proof}

 \begin{thm}[{\bf Theorem 3.1, \cite{yuan}}]\label{thm1}
  Let $r \geq 1$ and $n \geq 1$ be positive integers. Let $\psi, L_1,\ldots, L_r \in \F[x]$ be $q$-polynomials,
  $g \in \Fn[x], h_1,\ldots, h_r \in \F[x]$ and $\delta_1,\ldots, \delta_r \in \Fn$ such that $\psi(\delta_i) \in \F$
  and $h_i(\psi(\Fn)) \subseteq \F$. Then
  $$
  f(x) = g(\psi(x)) + \sum_{i=1}^r (L_i(x) + \delta_i)h_i(\psi(x))
  $$
  is a permutation polynomial of $\Fn$ if and only if the following two conditions hold.
  
  (i) $\bar{f}(x) := \psi(g(x)) + \sum_{i=1}^r (L_i(x) + \psi(\delta_i) ) h_i(x)$ permutes $\psi(\Fn)$; and
  
  (ii) for any $y \in \psi(\Fn)$, $\varphi_y(x) := \sum_{i=1}^r L_i(x) h_i(y)$ permutes $\ker(\psi)$.
  \end{thm}

  \begin{thm}\label{inv1}
   Using the same notations as in Theorem \ref{thm1}, assume $f$ permutes $\Fn$, and let $\bar{f}^{-1}$
   be the inverse of $\bar{f}|_{\psi(\Fn)}$. Then if $x \in \Fn$ is such that 
   $\ker(\varphi_{\bar{f}^{-1}(\psi(x))}) \cap \psi(S_\psi) = \{0\}$, the preimage of $x$ under $f$ is given by
   \begin{equation*}
\begin{split}
 f^{-1}(x) =
  \bar{f}^{-1}(\psi(x)) +    \varphi_{\bar{f}^{-1}(\psi(x))}^{-1}|_{S_\psi} 
  \Biggr( x & -  \psi(x) - g\left(\bar{f}^{-1}(\psi(x))\right) - \sum_{i=1}^r \delta_i h_i\left(\bar{f}^{-1}(\psi(x))\right)\\
   & + \psi\Biggr( g\left(\bar{f}^{-1}(\psi(x))\right) + \sum_{i=1}^r \delta_i h_i\left(\bar{f}^{-1}(\psi(x))\right) \Biggr) \Biggr).
 \end{split}
 \end{equation*}
 Furthermore, if $x \in \Fn$ is such that $\ker(\varphi_{\bar{f}^{-1}(\psi(x))}) \cap \psi(\Fn) = \{0\}$, then $\varphi_{\bar{f}^{-1}(\psi(x))}$
 permutes $\Fn$, and the preimage of $x$ under $f$ is given by
 $$
 f^{-1}(x) = \varphi_{\bar{f}^{-1}(\psi(x))}^{-1} \left( x - g\left(\bar{f}^{-1}(\psi(x))\right) - \sum_{i=1}^r \delta_i h_i\left(\bar{f}^{-1}(\psi(x))\right)    \right).
 $$
  \end{thm}
 
\begin{proof}
  As each $L_i$ is $q$-linear over $\F$ and each $h_i(y) \in \F$ for any $y \in \psi(\Fn)$, then 
  $\varphi_y$ is also $q$-linear over $\F$ for any such $y$. Then $\varphi_y \circ \psi = \psi \circ \varphi_y$.
  Of course, $L_i \circ \psi = \psi \circ L_i$ as well.
  Thus we apply Lemma \ref{i0} with $F = \phi_\psi \circ f \circ \phi_\psi^{-1}$. As before we let 
  $y := \psi(x)$ and $z := x - \psi(x)$. We have
\begin{align*}
  Y &:= \psi(f(x)) = \psi(g(y)) + \sum_{i=1}^r (L_i(y) + \psi(\delta_i))h_i(y) = \bar{f}(y) \text{; and}\\
  Z &:= f(x) - \psi(f(x)) = g(y) - \psi(g(y)) + \sum_{i=1}^r(L_i(z) + \delta_i - \psi(\delta_i))h_i(y)\\
   &= \varphi_y(z) + g(y) + \sum_{i=1}^r \delta_i h_i(y) - \psi \left(g(y) + \sum_{i=1}^r \delta_i h_i(y)   \right).
  \end{align*}
 Hence $F$ is defined by $F(y,z) = (Y,Z)$. If for $y = \bar{f}^{-1}(Y)$, $\ker(\varphi_y) \cap \psi(S_\psi) = \{0\}$,
 then $\varphi_y$ is invertible on $S_\psi$ by Lemma \ref{i2} since $\ker(\varphi_y) \cap \ker(\psi) = \{0\}$ additionally (because $f$ is a permutation). 
 For such $y$ we obtain
 $$
 z = \varphi_y^{-1}|_{S_\psi} \left( Z - g(y) - \sum_{i=1}^r \delta_i h_i(y) + \psi\left( g(y) + \sum_{i=1}^r \delta_i h_i(y)  \right)     \right).
 $$
 Hence
 \begin{equation*}
\begin{split}
 F^{-1}(Y,Z) =
  \Biggr(   \bar{f}^{-1}(Y),    \varphi_{\bar{f}^{-1}(Y)}^{-1}|_{S_\psi} 
  \Biggr( Z & - g\left(\bar{f}^{-1}(Y)\right) - \sum_{i=1}^r \delta_i h_i\left(\bar{f}^{-1}(Y)\right)\\
   & + \psi\Biggr( g\left(\bar{f}^{-1}(Y)\right) + \sum_{i=1}^r \delta_i h_i\left(\bar{f}^{-1}(Y)\right) \Biggr)  \Biggr) \Biggr).
  \end{split}
 \end{equation*}
 Since $f^{-1} = \phi_\psi^{-1} \circ F^{-1} \circ \phi_\psi$, letting $Y = \psi(x)$ and $Z = x - \psi(x)$, we obtain
 \begin{equation*}
\begin{split}
 f^{-1}(x) =
  \bar{f}^{-1}(\psi(x)) +    \varphi_{\bar{f}^{-1}(\psi(x))}^{-1}|_{S_\psi} 
  \Biggr( x & -  \psi(x) - g\left(\bar{f}^{-1}(\psi(x))\right) - \sum_{i=1}^r \delta_i h_i\left(\bar{f}^{-1}(\psi(x))\right)\\
   & + \psi\Biggr( g\left(\bar{f}^{-1}(\psi(x))\right) + \sum_{i=1}^r \delta_i h_i\left(\bar{f}^{-1}(\psi(x))\right) \Biggr) \Biggr)
 \end{split}
 \end{equation*}
 whenever $x \in \Fn$ is such that $\ker(\varphi_{\bar{f}^{-1}(\psi(x))}) \cap \psi(S_\psi) = \{0\}$.
 
 If $y \in \psi(\Fn)$ is such that $\varphi_{\bar{f}^{-1}(y)}$ permutes $\psi(\Fn)$, then $\varphi_{\bar{f}^{-1}(y)}$ permutes $\Fn$
 by Lemma \ref{corlem}. We claim that $\bar{f}^{-1}(y)$ satisfies
 $$
 \bar{f}^{-1}(y) = \varphi^{-1}_{\bar{f}^{-1}(y)} \left( y - \psi\left(g\left(\bar{f}^{-1}(y) \right) + \sum_{i=1}^r \delta_i h_i\left(\bar{f}^{-1}(y)   \right) \right)  \right).
 $$
 Indeed, 
 \begin{equation*}
  \begin{split}
  \bar{f}^{-1}(\bar{f}(y)) &= \varphi^{-1}_y \left( \bar{f}(y) - \psi\left(g(y) + \sum_{i=1}^r \delta_i h_i(y)  \right)  \right)\\
  &= \varphi^{-1}_y \left( \psi(g(y)) + \sum_{i=1}^r (L_i(y) + \psi(\delta_i) ) h_i(y) -  \psi\left(g(y) + \sum_{i=1}^r \delta_i h_i(y)  \right)  \right)\\
  &= \varphi^{-1}_y \left(\sum_{i=1}^r L_i(y) h_i(y)   \right)
  = \varphi_y^{-1}\left( \varphi_y(y) \right)\\
  &= y.
  \end{split}
  \end{equation*}
  Similarly one can show that $\bar{f}\left( \bar{f}^{-1}(y)  \right) = y$.
 Now using the fact that $\varphi^{-1}_{\bar{f}^{-1}(y)}$ is additive, we get the result.
 \end{proof}
 
 \begin{cor}[See Corollary 3.3, \cite{yuan}]\label{TcorYuan}
  Let $r \geq 1$ and $n \geq 1$ be positive integers. Let $L_1,\ldots, L_r \in \F[x]$ be $q$-polynomials,
  $g \in \Fn[x], h_1,\ldots, h_r \in \F[x]$, and $\delta_1,\ldots, \delta_r \in \Fn$ . If
  $$
  f(x) = g(T(x)) + \sum_{i=1}^r (L_i(x) + \delta_i)h_i(T(x))
  $$
  is a permutation of $\Fn$, then $\bar{f}(x) := T(g(x)) + \sum_{i=1}^r (L_i(1)x  + T(\delta_i))h_i(x)$ is a permutation of $\F$, and
  $\varphi_y(x) := \sum_{i=1}^r L_i(x) h_i(y)$ is a permutation of $\ker(T)$ for any $y \in \F$. 
  
  (i) If $p \mid n$ or $x \in \Fn$ is such that $\sum_{i=1}^r L_i(1) h_i(\bar{f}^{-1}(T(x))) \neq 0$ (equivalently, $\varphi_{\bar{f}^{-1}(T(x))}$ permutes $\F$), 
  then $\varphi_{\bar{f}^{-1}(T(x))}$ permutes $\Fn$, and the preimage of $x$ under $f$ is given by
  $$
  f^{-1}(x) = \varphi_{\bar{f}^{-1}(T(x))}^{-1} \left( x - g\left(\bar{f}^{-1}(T(x))\right) - \sum_{i=1}^r \delta_i h_i\left(\bar{f}^{-1}(T(x))\right)    \right).
  $$
  
  (ii) If $p \nmid n$, the inverse of $f$ on $\Fn$ is given by
  \begin{equation*}
\begin{split}
 f^{-1}(x) =
  n^{-1}\bar{f}^{-1}(T(x)) + \\  +  \varphi_{\bar{f}^{-1}(T(x))}^{-1}|_{\ker(T)}
 \Biggr( x & -  n^{-1}T(x) - g\left(\bar{f}^{-1}(T(x))\right) - \sum_{i=1}^r \delta_i h_i\left(\bar{f}^{-1}(T(x))\right)\\
   & + n^{-1}T\Biggr( g\left(\bar{f}^{-1}(T(x))\right) + \sum_{i=1}^r \delta_i h_i\left(\bar{f}^{-1}(T(x))\right) \Biggr) \Biggr).
 \end{split}
 \end{equation*}
  \end{cor}
  
  \begin{proof}
   That $\bar{f}, \varphi_y$, are permutations of $\F$ and $\ker(T)$, respectively, follows from Theorem \ref{thm1}.
   
   (i) If $p \mid n$, then $\F \subseteq \ker(T)$ since $T(c) = nc = 0$
   for any $c \in \F$. But since $\varphi_{y}$ permutes $\ker(T)$ for any $y \in \F$, and $\varphi_y(\F) \subseteq \F$, it follows that $\varphi_y$ permutes $\F$.
   Now Lemma \ref{corlem} implies $\varphi_{y}$ permutes $\Fn$ for any $y \in \F$. That $\varphi_y$ permuting $\F$ is equivalent with $\sum_{i=1}^r L_i(1) h_i(y) \neq 0$
   follows from the fact that each $L_i$ is a $q$-polynomial over $\F$. The rest follows from Theorem \ref{inv1}.
   
   (ii) If $p \nmid n$, then $\ker(T) = \{x - n^{-1}T(x) \mid x \in \Fn \}$; thus if we let $\psi = n^{-1}T$, we get $S_{\psi} = \ker(\psi) = \ker(T)$.
   Note that if we define the polynomial $H(x) := f(n^{-1} x)$, then we obtain an instance of Theorem \ref{thm1} with $\psi = n^{-1}T$,
   $\bar{H} = n^{-1}\bar{f}$ and 
   $\varphi_{y}^{(H)} = n^{-1}\varphi_y$, from which it follows that $\bar{H}^{-1}(n^{-1}T(x)) = \bar{f}^{-1}(T(x))$,
   and $\varphi_y^{-1(H)}|_{\ker(T)} = n \varphi_y^{-1}|_{\ker(T)}$. Now since $f(x) = H(nx)$, then $f^{-1}(x) = n^{-1} H^{-1}(x)$, and the result follows from Theorem \ref{inv1}.
  \end{proof}

\begin{cor}[See Corollary 3.2, \cite{yuan}]\label{key0}
  Let $L_1,L_2,L_3 : \Fn \rightarrow \Fn$ be $q$-polynomials over $\F$. Let $w \in \Fn[x]$ such that $w(L_3(\Fn)) \subseteq \F$, and assume that
  $$
  f(x) = L_1(x) + L_2(x) w(L_3(x))
  $$
  permutes $\Fn$. Let $\bar{f}(x) = L_1(x) + L_2(x) w(x)$, and for $y \in L_3(\Fn)$, let $\varphi_y(x) = L_1(x) + L_2(x) w(y)$.
  Then if
  $x \in \Fn$ is such that $\ker(\varphi_{\bar{f}^{-1}(L_3(x))}) \cap L_3(S_{L_3}) = \{0\}$, the preimage of $x$ under $f$ is given by 
  $$
  f^{-1}(x) = \bar{f}^{-1}(L_3(x)) + \varphi^{-1}_{\bar{f}^{-1}(L_3(x))}|_{S_{L_3}}
  \left(x - L_3(x)  \right).
  $$
  Furthermore, if $x \in \Fn$ is such that $\varphi_{\bar{f}^{-1}(L_3(x))}$ permutes $L_3(\Fn)$, then $\varphi_{\bar{f}^{-1}(L_3(x))}$ permutes $\Fn$, and the preimage of $x$ under $f$ is given by
  $$
  f^{-1}(x) = \varphi^{-1}_{\bar{f}^{-1}(L_3(x))}(x).
  $$
  \end{cor}

 \begin{proof}
  Here $r=2$, $g = 0$, $h_1 = 1$, $h_2 = w$, $\delta_1 = \delta_2 = 0$, and $\psi = L_3$. In the case of $\varphi_{\bar{f}^{-1}(y)}$ permuting $L_3(\Fn)$ for $y \in L_3(\Fn)$, it permutes $\Fn$
  since it additionally permutes $\ker(L_3)$ because $f$ is a permutation. Then since $\bar{f}(y) = \varphi_y(y)$, we get $\bar{f}^{-1}(y) = \varphi^{-1}_{\bar{f}^{-1}(y)}(y)$, from which the result follows.
 \end{proof}
 
 The proof of Theorem \ref{nice3} is essentially the same as the proof of Theorem \ref{key0}, where the linear equation in part (b) follows from Corollary \ref{cor:IdempotentLinearEqn}.

\begin{thm}[{\bf Theorem 3, \cite{coulter}}]\label{key1}
Let $q = p^m$ be a prime power, let $g \in \mathbb{F}_q[x]$, let $H \in \F[x]$ be additive, and let $f(x) = H(x) + xg(T(x))$.
Then $f(x)$ is a permutation polynomial of $\mathbb{F}_{q^n}$
if and only if the following two conditions hold.

(i) For any $y \in \mathbb{F}_q$ and any $x \in \mathbb{F}_{q^n}$ we have $\varphi_y(x) := H(x) + xg(y) = 0$ and $T(x) = 0$ if and only if $x=0$.

(ii) $\bar{f}(x) := H(x) + xg(x)$ is a permutation polynomial of $\mathbb{F}_q$.
 \end{thm}
 
 \begin{cor}\label{cor:key1}
  Using the same notations as in Theorem \ref{key1}, we have that $\varphi_y(x) = H(x) + xg(y)$ is a permutation of $\ker(T)$, for any $y \in \F$. 
  
  (i) If $p \mid n$ or $x \in \Fn$ is such that $\varphi_{\bar{f}^{-1}(T(x))}$ permutes $\F$, then $\varphi_{\bar{f}^{-1}(T(x))}$
  permutes $\Fn$, and the preimage of $x$ under $f$ is given by
  $$
  f^{-1}(x) = \varphi^{-1}_{\bar{f}^{-1}(T(x))}(x).
  $$
  
  (ii) If $p \nmid n$, then the inverse of $f$ on $\Fn$ is given by
  $$
  f^{-1}(x) = n^{-1}\bar{f}^{-1}(T(x)) + \varphi^{-1}_{\bar{f}^{-1}(T(x))}|_{\ker(T)} \left(x - n^{-1}T(x)  \right).
  $$
  \end{cor}
 
 \begin{proof}
  It follows directly from Corollary \ref{TcorYuan}, \ref{key0}, and Theorem \ref{key1}.
 \end{proof}

 Note that in Theorem \ref{key1} if we let $q$ be even, $n$ be odd, $H(x) = ax^2$, and $g(x) = L(x) + ax$, then we obtain
 the permutation in (3) of the Introduction. 
 In fact, Theorem \ref{key1} and Corollary \ref{cor:key1} are special cases of Theorem \ref{nice3}.

 \section{Compositional inverse of a general class}
 

 In this section we give the compositional inverse of a class of permutation polynomials generalizing that in (2) and (3) of the Introduction, the compositional inverses
 of which were given in \cite{coulter1} and \cite{wu}, respectively. First in Lemma \ref{lem1} we give a compositional inverse over $\ker(T)$ of a simple $p$-polynomial inducing a permutation
 of $\ker(T)$. Then, using this result, we give, in Theorem \ref{generalization} and Corollary \ref{cor:compositional_inverse}, the preimages and compositional inverse, respectively,
 of the more general class of permutation polynomials.
 Similarly as done in \cite{wu}, our result is left as an expression in terms of the inverse of $\bar{f}$ over the subspace $\F$. 
  
 \begin{lem}\label{lem1}
   Let $q = p^m$ be a power of a prime $p$, let $c \in \mathbb{F}_q^*$, let $P_c(x) = x^p + cx$,
   and for $0 \leq j\leq m-1$, let $ a_j = (-1)^j c^{-(p^{j+1} - 1)/(p-1)}$.
   Then $P_c$ permutes $\ker(T) = \{\beta^q - \beta \mid \beta \in \Fn \}$ if and only if 
   $c$ belongs to any of the following two cases, 
   for which a corresponding compositional inverse on $\ker(T)$ is given as follows.\\ 
  {\bf Case 1:} $c^{(p^m-1)/(p-1)} = (-1)^m$ and $p \nmid n$.
   
   For any $\delta \in \mathbb{F}_p$, 
   $$
   P_c^{-1}(x) = \dfrac{1}{n} \sum_{j=0}^{m-1}\sum_{k=0}^{n-1}a_j (\delta - k)x^{p^{km+j}}.
   $$
   \\
  {\bf Case 2:} $c^{n(p^m-1)/(p-1)} \neq (-1)^{mn}$. Equivalently, $P_c$ permutes $\Fn$.
   
   The inverse of $P_c$ on $\Fn$, and hence on $\ker(T)$, is given by
   $$
   P_c^{-1}(x) = \dfrac{1}{c^{n(p^m-1)/(p-1)} + (-1)^{mn-1}} \sum_{i=0}^{mn-1}(-1)^i c^{\sum_{j= i+1}^{mn-1} p^j} x^{p^i}. 
   $$
   \end{lem}

  \begin{proof} 
  {\bf Case 1:} Assuming that $c^{(p^m-1)/(p-1)} = (-1)^m$, $p \nmid n$, and using the facts that $a_jc^{p^j} = -a_{j-1}$ for $j > 0$, $a_{m-1} = -1$, and $a_0 = c^{-1}$, we get 
   \begin{eqnarray*}
    P_c^{-1}\left(P_c(x) \right) 
    &=& n^{-1}\sum_{j=0}^{m-1}\sum_{k=0}^{n-1}a_j (\delta - k)\left(x^p + cx \right)^{p^{km+j}}\\
    &=& n^{-1}\sum_{j=0}^{m-1}\sum_{k=0}^{n-1}a_j (\delta - k)\left(x^{p^{km+j+1}} + c^{p^j}x^{p^{km+j}} \right)\\
    &=& n^{-1} \sum_{j=1}^{m}\sum_{k=0}^{n-1}a_{j-1}(\delta - k)x^{p^{km+j}} + n^{-1} \sum_{j=0}^{m-1}\sum_{k=0}^{n-1}a_{j}(\delta - k)c^{p^j}x^{p^{km+j}}\\
    &=& n^{-1}\sum_{k=0}^{n-1}(k- \delta)x^{p^{(k+1)m}} + n^{-1}\sum_{k=0}^{n-1}(\delta - k)x^{p^{km}}\\
    && \ \ \ \ + n^{-1}\sum_{j=1}^{m-1}\sum_{k=0}^{n-1}(\delta - k)(a_{j-1} + a_j c^{p^j})x^{p^{km+j}}\\
    &=& n^{-1}\sum_{k=1}^{n}(k-1-\delta)x^{p^{km}} + n^{-1}\sum_{k=0}^{n-1}(\delta - k)x^{p^{km}}\\
    &=& n^{-1}(n-1)x - n^{-1} \sum_{k=1}^{n-1}x^{p^{km}}\\
    &=& n^{-1}(n-1)x + n^{-1}(x - T(x))\\
    &=& x - n^{-1} T(x)\\
    &=& x,
    \end{eqnarray*}
    for all $x \in \ker(T)$. Similarly, using the fact that $ca_j = -a_{j-1}^p$ for $j > 0$, one can show that $P_c(P_c^{-1}(x)) = x$ for all $x \in \ker(T)$.
    Note that $P^{-1}_c$ must be unique upon reduction modulo $T$; we may set for instance $\delta = 0$. Moreover, the existence of the inverse $P_c^{-1}$ on $\ker(T)$
    is equivalent with $P_c$ permuting $\ker(T)$ as maps are invertible on a set if and only if they induce bijections of that set.
   
   In the case when $p \mid n$, we show that $P_c$ does not permute $\ker(T)$ for any such $c$. Note that $T(k) = nk = 0$ for any $k \in \F$;
   thus $\F \subseteq \ker(T)$. Then since $P_c(\F) \subseteq \F$ additionally, it suffices to show that $P_c$ does not permute $\F$. To do this we show that the determinant of 
   $m \times m$ associate Dickson matrix, $D_{P_c}$, is zero. Indeed, 
   $$
   D_{P_c} = \left( \begin{array}{cccccc}
                 c & 1 & 0 & 0 &\cdots & 0\\
                 0 & c^p & 1 & 0 & \cdots & 0\\
                 \vdots & \vdots & & & & \vdots\\
                1 & 0 & 0 & 0 & \cdots & c^{p^{m-1}}
                \end{array} \right); 
   $$
  thus, applying the cofactor expansion in the first column, we get $\det(D_{P_c}) = c^{(p^m-1)/(p-1)} + (-1)^{m-1} = 0$, by assumption. 
  This completes the proof of Case 1. Note that we have also obtained that
  $P_c$ is a permutation of $\F$ if and only if $c^{(p^m-1)/(p-1)} \neq (-1)^{m}$.
   
  
  {\bf Case 2:} Otherwise assume $c^{(p^m-1)/(p-1)} \neq (-1)^m$. We show that $P_c$ is a permutation of 
  $\ker(T)$ if and only if it is a permutation of $\Fn$ which happens if and only if $c^{n(p^m-1)/(p-1)} \neq (-1)^{mn}$. 
  Now, from the proof of Case 1 we know that in this case $P_c$ is a permutation of $\F$. Since 
  $p$-polynomials over $\F$ commute with $q$-polynomials over $\mathbb{F}_p$, we have $P_c \circ T = T \circ P_c$.
  Thus we can apply Lemma \ref{corlem} to obtain that $P_c$ is a permutation of $\ker(T)$ if and only if
  it permutes $\Fn$ (since $T(\Fn) = \F$ additionally). Thus, similarly as done in Case 1, we compute the determinant of the $mn \times mn$ associate Dickson matrix. We have
  $$
   D_{P_c} = \left( \begin{array}{cccccc}
                 c & 1 & 0 & 0 &\cdots & 0\\
                 0 & c^p & 1 & 0 & \cdots & 0\\
                 \vdots & \vdots & & & & \vdots\\
                1 & 0 & 0 & 0 & \cdots & c^{p^{mn-1}}
                \end{array} \right); 
   $$
   hence, using the fact that $c \in \F$, we get $\det(D_{P_c}) = c^{\sum_{i=0}^{mn-1} p^i} + (-1)^{mn-1} = c^{n(p^m-1)/(p-1)} + (-1)^{mn-1}$. 
   Therefore $P_c$ is a permutation of $\Fn$ if and only if $c^{n(p^m-1)/(p-1)} \neq (-1)^{mn}$.
   Here we will use Proposition \ref{LinInv} to obtain an expression for $P_c^{-1}$. 
   For this we need the $(i,0)$-th cofactors, $0 \leq i \leq mn-1$, of $D_{P_c}$, which are given by $\bar{a}_i = (-1)^i c^{\sum_{j = i+1}^{mn-1} p^j }$.
   Now Proposition \ref{LinInv} gives the result. Indeed,
   \begin{equation*}
    \begin{split}
     P_c^{-1}\left(P_c(x)\right) &= \sum_{i=0}^{mn-1} \dfrac{(-1)^i c^{\sum_{j= i+1}^{mn-1} p^j} \left(x^p + cx  \right)^{p^i}}{c^{n(p^m-1)/(p-1)} + (-1)^{mn-1}} \\
     &= \sum_{i=0}^{mn-1} \dfrac{(-1)^i c^{\sum_{j= i+1}^{mn-1} p^j} \left(x^{p^{i+1}} + c^{p^i}x^{p^i}  \right)}{c^{n(p^m-1)/(p-1)} + (-1)^{mn-1}}\\
     &= \sum_{i=1}^{mn} \dfrac{(-1)^{i-1} c^{\sum_{j= i}^{mn-1} p^j} x^{p^{i}} }{c^{n(p^m-1)/(p-1)} + (-1)^{mn-1}} + \sum_{i=0}^{mn-1} \dfrac{(-1)^i c ^{\sum_{j=i}^{mn-1} p^j} x^{p^i} } {c^{n(p^m-1)/(p-1)} + (-1)^{mn-1}}\\      
     &= x + \sum_{i=1}^{mn-1} \dfrac{\left( (-1)^{i-1} c^{\sum_{j= i}^{mn-1} p^j} + (-1)^{i} c^{\sum_{j= i}^{mn-1} p^j}  \right) x^{p^i} }{c^{n(p^m-1)/(p-1)} + (-1)^{mn-1}}\\
     &= x.
    \end{split}
    \end{equation*}
    Similarly one can show that $P_c(P_c^{-1}(x)) = x$. At this point we have exhausted all the possibilities for $c$ and so the proof is complete.
    \end{proof}

    \begin{rmk}
    Of course one can use Theorem \ref{thm:LinearizedInverse} to obtain the result in Case 1, 
    as we have done in Example \ref{eg:IdempotentLinearEqn} with $\delta = n-1$ there.
    As a way of comparison, and in the style of Remark 3.4 in \cite{wu}, here we give an explanation of the ``direct" method the authors used prior to obtaining the aforementioned results.
     First note that $P_c(\ker(T)) \subseteq \ker(T)$. As the set of $p$-polynomials over $\F$ permutating a subspace of $\Fn$ forms a group under composition,
  assume $P_c^{-1}(x) = \gamma \sum_{i=0}^{mn-1} d_i x^{p^i}$ for some $\gamma \in \mathbb{F}_p^*$ and some elements $d_i \in \F$, whenever $P_c$ is invertible on $\ker(T)$.
   Write 
  $
  P_c^{-1}(x) = \gamma \sum_{j=0}^{m-1} \sum_{k=0}^{n-1} d_{km + j} x^{p^{km+j}}.
  $
 For any $x \in \ker(T)$, and using the fact that $x - T(x) = -x^{p^m} -x^{p^{2m}} - \cdots - x^{p^{(n-1)m}}$, assume 
  \begin{align*}
 P_c^{-1}(P_c(x)) 
 &= 
 \gamma \sum_{j=0}^{m-1} \sum_{k=0}^{n-1} d_{km + j} \left( x^p + cx \right)^{p^{km+j}} \\
&= 
\gamma   (d_{nm-1} + d_0 c)x + \gamma \sum_{k=1}^{n-1}(d_{km-1} + d_{km} c)x^{p^{km}}
 + \gamma \sum_{j=1}^{m-1}\sum_{k=0}^{n-1}\left(d_{km+j-1} + d_{km+j}c^{p^j}\right)x^{p^{km+j}} \\
&=
\gamma (d_{nm-1} + d_0 c + 1)x - \gamma T(x)\\
&=
\gamma (d_{nm-1} + d_0 c + 1)x \\
&=
x.
\end{align*}
We must have $d_{nm-1} + d_0 c + 1 \neq 0$ and the system of equations
$$
\begin{cases} d_{i-1} + d_{i}c^{p^i} = 0,\  if \ m \nmid i; \\ d_{i-1} + d_{i} c = -1,\ if \ m \mid i, \end{cases}
$$
for $0 < i < mn$, with solution given by $d_{km+j} = a_j c d_{km}  = a_j b_k$, where
$$
b_k := (-a_{m-1})^k \delta - \sum_{l=0}^{k-1}(-a_{m-1})^l
$$
with $\delta := c d_0$ to be determined, for $0 \leq j \leq m-1$ and $0 \leq k \leq n-1$.
Note that if $\delta \in \mathbb{F}_p$, then each $b_k \in \mathbb{F}_p$ 
as $a_{m-1} = -N_{q|p}(-c^{-1}) \in \mathbb{F}_p$, 
where $N_{q|p}: \F \rightarrow \mathbb{F}_p$ is the absolute {\em norm} function given by $N_{q|p}(y) = y^{(q-1)/(p-1)}$.
Hence if $d_{nm-1} + d_0 c + 1 \neq 0$, we set $\gamma = (d_{nm-1} + d_0 c + 1)^{-1} = (a_{m-1}b_{n-1} + \delta + 1)^{-1} \in \mathbb{F}_p^*$.
It is left to the reader to check that $P_c^{-1}(P_c(x)) = P_c(P^{-1}_c(x)) = x$ on $\ker(T)$ if $\delta \in \mathbb{F}_p$ 
and $a_{m-1} b_{n-1} + \delta + 1 \neq 0$,
in which case the inverse of $P_c$ on $\ker(T)$ is given by
$$
P_c^{-1}(x) = \left(a_{m-1} b_{n-1} + \delta + 1  \right)^{-1} \sum_{j=0}^{m-1} \sum_{k=0}^{n-1} a_j b_k x^{p^{km + j}}.
$$
Now for Case 1, if we assume $c^{(p^m-1)/(p-1)} = (-1)^m$, then $a_{m-1} = -1$; hence $b_k = \delta - k$ and $a_{m-1}b_{n-1} + \delta + 1= -(\delta - n + 1) + \delta + 1= n \neq 0$
   since $p \nmid n$ by assumption. 
    \end{rmk}

  \begin{cor}[{\bf Lemma 3.3, \cite{wu}}]\label{cor:lem1}
   Let $q = 2^m$ and $n$ be odd. Let $P_c(x) = x^2 + cx$ for any $c \in \mathbb{F}_q^*$. Then $P_c$ can induce a permutation of $\ker(T)$ and one of 
   the polynomials that can induce its inverse map is
   $$
   P_c^{-1}(x) = \sum_{j=0}^{m-1}c^{-(2^{j+1}-1)}\left(\sum_{k=0}^{\frac{n-1}{2}} x^{q^{2k}}   \right)^{2^j}.
   $$
  \end{cor}
  
  \begin{proof}
   As $2 \nmid n$ and $c^{2^m-1} = 1$, this is Case 1 of Lemma \ref{lem1} with $\delta = 1$.
  \end{proof}
  
 Denote by $N_{q|p} : \F \rightarrow \mathbb{F}_p$ the absolute {\em norm} function given by $N_{q|p}(y) = y^{(q-1)/(p-1)}$.
 
   \begin{thm}\label{generalization}
  Let $q = p^m$ be a power of a prime $p$, let $n$ be a positive integer, let $a \in \F^*$, and assume that $g \in \F[x]$ is such that $\bar{f}(x) := ax^p + x g(x)$ induces a permutation of 
  $\F$ and $\varphi_y(x) := ax^p + xg(y)$ induces a permutation of $\ker(T) = \{ \beta^q - \beta \mid \beta \in \Fn \}$, for each $y \in \F$. Then 
  $$f(x) := ax^p + x g(T(x))$$ 
  induces a permutation of $\Fn$. Let $\bar{f}^{-1}$ be the inverse of the permutation $\bar{f}|_{\F}$.
  
  (a) If $x \in \Fn$ is such that $g(\bar{f}^{-1}(T(x))) = 0$, the preimage of $x$ under $f$ is given by
  $$
  f^{-1}(x) = \left( \dfrac{x}{a} \right)^{q^n/p}.
  $$
  
  Otherwise, if $p \mid n$ or $x \in \Fn$ is such that $\varphi_{\bar{f}^{-1}(T(x))}$ permutes $\F$ (equivalently,\\
  $N_{q|p}(g(\bar{f}^{-1}(T(x)))/a) \neq (-1)^m$), then $\varphi_{\bar{f}^{-1}(T(x))}$ permutes $\Fn$ and
  the preimage of $x$ under $f$ is given by
  $$
  f^{-1}(x) = \sum_{i=0}^{mn-1} \dfrac{(-1)^i a^{(p^i-1)/(p-1)} g\left(\bar{f}^{-1}\left(T(x)\right)\right) ^{\sum_{j=i+1}^{mn-1} p^j}}{N_{q|p}\left( g\left(\bar{f}^{-1}\left(T(x)\right)\right)^n\right) - N_{q|p}\left( (-a)^{n} \right)} x^{p^i}.
  $$
   
   (b) Otherwise (if $p \nmid n$ and $x \in \Fn$ is such that $N_{q|p}(g(\bar{f}^{-1}(T(x)))/a) = (-1)^m$), the preimage of $x$ under $f$ is given by
  $$
  f^{-1}(x) = n^{-1} \left( \bar{f}^{-1}(T(x)) -\sum_{j=0}^{m-1}\dfrac{(-1)^j a^{(p^j-1)/(p-1)}}{g\left( \bar{f}^{-1}(T(x)) \right)^{(p^{j+1} - 1)/(p-1)} }\left(\sum_{k=1}^{n-1} k \left( x^{q^k} - n^{-1}T(x)   \right)   \right)^{p^j}  \right).
  $$
 \end{thm}
 
 \begin{proof}
  The fact that $f$ permutes $\Fn$ under the assumptions follows from Theorem \ref{key1} with $H(x) = ax^p$.
  
  (a) In this case we have $f^{-1}(x) = \varphi_{\bar{f}^{-1}(T(x))}^{-1}(x)$ by Corollary \ref{cor:key1}. 
  If $y \in \F$ is such that $g(y) = 0$, then $\varphi_y(x) = ax^p$ is a permutation of $\Fn$, and $\varphi_y^{-1}(x) = (x/a)^{q^n/p}$.
  Otherwise, noting that $\varphi_y(x) = a P_{g(y)/a}(x)$, where $P_c(x) = x^p + cx$ from Lemma \ref{lem1}, and substituting $y$ with $\bar{f}^{-1}(T(x))$, we get $f^{-1}(x) = P^{-1}_{g(\bar{f}^{-1}(T(x)))/a} (x/a)$, 
  where $P^{-1}_c$, with $c = g(\bar{f}^{-1}(T(x)))/a$, is given in Case 2 of Lemma \ref{lem1}. Thus we obtain
  $$
  f^{-1}(x) = \dfrac{\sum_{i=0}^{mn-1}(-1)^i\left( \dfrac{g\left( \bar{f}^{-1}(T(x)) \right)}{a} \right)^{\sum_{j=i+1}^{mn-1} p^j} \left( \dfrac{x}{a}  \right)^{p^i}}   {N_{q|p}\left( \dfrac{g\left( \bar{f}^{-1}(T(x)) \right)}{a} \right)^n + (-1)^{mn-1}}.  
  $$
  Now since $N_{q|p}(a)^n (-1)^{mn-1} = - N_{q|p}(-a)^n$ 
  and $N_{q|p}(a)^n a^{- \sum_{j=i}^{mn-1}p^j} = a^{\sum_{j=0}^{mn-1} p^j - \sum_{j=i}^{mn-1}p^j} = a^{\sum_{j=0}^{i-1} p^j} = a^{(p^i-1)/(p-1)}$, the result follows.

(b) Assume that $y \in \F$ is such that $N_{q|p}(g(y)/a) = (-1)^m$. Then $\varphi_y(x) = ax^p + xg(y) = a(x^p + xg(y)/a) = aP_{g(y)/a}(x)$ where $P_{c_y}$, $c_y = g(y)/a \in \F^*$, is a permutation of $\ker(T)$ in Case 1 of Lemma \ref{lem1}. Hence $\varphi^{-1}_y|_{\ker(T)}(x) = P^{-1}_{g(y)/a}(x/a)$, where
$P^{-1}_{g(y)/a}$ is given in Case 1 of Lemma \ref{lem1}. Then substituting $y$ with $\bar{f}^{-1}(T(x))$ in $\varphi^{-1}_y|_{\ker(T)}$ and then using Corollary \ref{cor:key1} together with Case 1 of Lemma \ref{lem1} with 
$c = g(\bar{f}^{-1}(T(x)))/a$ and $\delta = 0$, we get 
\begin{equation*}
\begin{split}
f^{-1}(x) &= n^{-1}\bar{f}^{-1}(T(x)) + P_{g(\bar{f}^{-1}(T(x)))/a}^{-1}\left(\frac{x - n^{-1}T(x)}{a} \right)\\
&= n^{-1}\bar{f}^{-1}(T(x))\\
& \ \ \ \ - n^{-1}\sum_{j=0}^{m-1}\sum_{k=1}^{n-1}k(-1)^j \left( \dfrac{a}{g\left(\bar{f}^{-1}(T(x))\right)} \right)^{\frac{p^{j+1} - 1}{p-1}}\left( \dfrac{x-n^{-1}T(x)}{a} \right)^{q^k p^j}\\
&= 
n^{-1}\bar{f}^{-1}(T(x))\\
& \ \ \ \ - n^{-1}\sum_{j=0}^{m-1}\sum_{k=1}^{n-1}k(-1)^j \left( \dfrac{a}{g\left(\bar{f}^{-1}(T(x))\right)} \right)^{\frac{p^{j+1} - 1}{p-1}}a^{-p^j}\left( x^{q^k}-n^{-1}T(x) \right)^{p^j}\\
&= n^{-1} \left( \bar{f}^{-1}(T(x)) -\sum_{j=0}^{m-1}\dfrac{(-1)^j a^{(p^j-1)/(p-1)}}{g\left( \bar{f}^{-1}(T(x)) \right)^{(p^{j+1} - 1)/(p-1)} }\left(\sum_{k=1}^{n-1} k \left( x^{q^k} - n^{-1}T(x)   \right)   \right)^{p^j}  \right),
\end{split}
\end{equation*}
as required.
\end{proof}

\begin{rmk}
 One may find several $g \in \F[x]$ such that $\varphi_y (x) = ax^p + xg(y) = a(x^p + x g(y) / a)$ induces a permutation of $\ker(T)$ for each $y \in \F$.
 For instance if we let $m$ be even and $n$ be such that $p \nmid n$ and $\gcd(n,p-1) = 1$, then 
 $\varphi_y(x) = a(x^p + x g(y)/a)$ induces a permutation of $\ker(T)$ for each $g \in \F[x]$ and each $y \in \F$. 
 Indeed, from Lemma \ref{lem1} it follows that $\varphi_y$ induces a permutation of 
 $\ker(T)$ if and only if $N_{q|p}(g(y)/a) = 1$ or $N_{q|p}(g(y)/a)^n \neq 1$. Noting that $\gcd(n,p-1) = 1$ implies that there does not exist an $n$-th root
 of unity in $\mathbb{F}_p \setminus\{1\}$, then if $N_{q|p}(g(y)/a) \neq 1$, we get $N_{q|p}(g(y)/a)^n \neq 1$ as required.
 For example we can pick $g(x) = x^{p-1}$, and if $a \neq -1$, then $\bar{f}(x) := ax^p + xg(x) = (a+1)x^p$ permutes $\F$, having inverse (on $\F$) given by
 $\bar{f}^{-1}(x) = (x/(a+1))^{q/p}$. As another example we can let $g(x) = x^{q-2}L(x)$ where $L \in \F[x]$ is any $p$-polynomial such that
 $\bar{f}|_{\F}(x) = ax^p + xg(x) = ax^p + L(x)$ is a permutation of $\F$. Then one can use for instance Proposition \ref{LinInv} to obtain the inverse of $\bar{f}|_{\F}$.
\end{rmk}

\begin{cor}\label{cor:compositional_inverse}
 Using the same notations and assumptions of Theorem \ref{generalization}, the compositional inverse of $f$ on $\Fn[x]$ is given by
 \begin{equation*}
  \begin{split}
  f^{-1}(x) &= \left( 1 - g\left( \bar{f}^{-1}(T(x))  \right)^{q-1} \right) \left( \dfrac{x}{a}  \right)^{q^n/p}\\
  &+ g\left( \bar{f}^{-1}(T(x)) \right)^{q-1}\left( \left(\dfrac{g\left( \bar{f}^{-1}(T(x)) \right)}{a}\right)^{(q-1)/(p-1)} - (-1)^m\right)^{p-1}\\
  & \cdot \sum_{i=0}^{mn-1} \dfrac{(-1)^i a^{(p^i-1)/(p-1)} g\left(\bar{f}^{-1}\left(T(x)\right)\right) ^{\sum_{j=i+1}^{mn-1} p^j}}{g\left(\bar{f}^{-1}\left(T(x)\right)\right)^{n(q-1)/(p-1)} - (-a)^{n(q-1)/(p-1)}} x^{p^i}\\
  &+ \left ( 1 - \left( \left(\dfrac{g\left( \bar{f}^{-1}(T(x)) \right)}{a}\right)^{(q-1)/(p-1)} - (-1)^m\right)^{p-1}\right)\\
  & \cdot
  n^{p-2} \left( \bar{f}^{-1}(T(x)) -\sum_{j=0}^{m-1}\dfrac{(-1)^j a^{(p^j-1)/(p-1)}}{g\left( \bar{f}^{-1}(T(x)) \right)^{(p^{j+1} - 1)/(p-1)} }\left(\sum_{k=1}^{n-1} k \left( x^{q^k} - n^{p-2}T(x)   \right)   \right)^{p^j}  \right).
  \end{split}
  \end{equation*}
\end{cor}

\begin{proof}
We put the results of Theorem \ref{generalization} (a) and (b) together. This is a step function where only one of the three terms is non-zero at a time. 
Denote $c_x := g(\bar{f}^{-1}(T(x)))/a$. 
The first term of the expression is non-zero only if $c_x = 0$ corresponding to the result in (a), while the second is non-zero only if
$c_x^{(q-1)/(p-1)} \neq 0, (-1)^m$, given in $(a)$ as well. The third term is non-zero only if 
$c_x^{(q-1)/(p-1)} = (-1)^m$ and $p \nmid n$, which is given in (b). Because $\varphi_{c_x}$ is a permutation of $\ker(T)$, 
then by Lemma \ref{lem1}
these are all the possibilities of $c_x$; hence we are done.
\end{proof}

In Theorem \ref{generalization}, if we let $p = 2$, $n$ be odd, $g(x) = x$, and $a \in \F \setminus \{0,1\}$, then we get 
$f(x) = ax^2 + xT(x) = x(T(x) + ax)$ in (2) of the Introduction, the compositional inverse of which was given in \cite{coulter1}. Similarly,
if instead we let $g(x) = L(x) + ax$ for some additive $L$ over $\F$ and some $a \in \F^*$,
then we get $f(x) = ax^2 + xg(T(x)) = x(L(T(x)) + aT(x) + ax)$ in (3) 
of the Introduction, with compositional inverse given in \cite{wu}. 
Thus these compositional inverses should follow from Corollary \ref{cor:compositional_inverse}.

  \section{Explicit compositional inverse of a second class}
  
  We further demonstrate the utility of Theorem \ref{inv0} by obtaining the explicit compositional inverse of a class of permutation polynomials generalizing 
  that of a linearized class given in Theorem 2.2 of \cite{wu_new}, where the corresponding inverse was obtained. Our result is presented in Theorem \ref{thm: simple proof}.
  See also Lemma \ref{lem: simple proof} which implies the aforementioned result by giving the compositional inverse of a more general class of linearized permutation
  polynomials.
  In particular, the method of our proof, which is an application of Theorem \ref{inv0}, 
  seems considerably less complicated than that employed in the aforementioned paper; the latter consisted of a direct application of Proposition \ref{LinInv} 
  and the computation of determinants. First we need the following two lemmas.

\begin{lem}\label{lem: special case}
Let $\alpha \in \Fn$ and denote $T_\alpha(x) := T(\alpha x)$. 
Then $\varphi(x) := x^q - x$ induces a bijection from $\ker(T_\alpha)$ to $\ker(T)$, in $\Fn$, if and only if $T(\alpha) \neq 0$. In the case that $T(\alpha) \neq 0$,
one of the polynomials, over $\Fn$, inducing the inverse map of $\varphi|_{\ker(T_\alpha)}$, is given by
$$
\varphi^{-1}|_{\ker(T)}(x) = T(\alpha)^{-1} \sum_{k=0}^{n-1} \sum_{j=0}^k \alpha^{q^j} x^{q^k}.
$$
\end{lem}

\begin{proof}
If we assume that $T(\alpha) = 0$, then 
$T_\alpha(1) = T(\alpha) = 0$ and $\varphi(1) = 0$; hence $1 \in \ker(\varphi) \cap \ker(T_\alpha)$ and
$\varphi|_{\ker(T_\alpha)}$ is not bijective. Now assume $T(\alpha) \neq 0$. 
Since $\ker(T) = \{x^q - x \mid x \in \Fn   \} = \im(\varphi)$, then $\varphi(\ker(T_\alpha)) \subseteq \ker(T)$. 
Note that $\im(T_\alpha) \subseteq \F$ (because $\im(T) = \F$).
Moreover, for any $c \in \F$, we have $c = cT(\alpha)^{-1}T(\alpha) = T(c T(\alpha)^{-1} \alpha) \in \im(T_\alpha)$; 
hence $\im(T_\alpha) = \im(T) = \F$. It follows that $|\ker(T_\alpha)| = |\ker(T)|$. 
To show that $\varphi|_{\ker(T_\alpha)}$ is injective, let $k \in \ker(\varphi) \cap \ker(T_\alpha)$.
In particular, $\varphi(k) = k^q - k = 0$ implies $k \in \F$. Then $T_\alpha(k) = T(\alpha k) = k T(\alpha) = 0$ gives $k = 0$ as required.
Thus $\varphi$ is bijective from $\ker(T_\alpha)$ to $\ker(T)$.
Next we prove that the given $\varphi^{-1}|_{\ker(T)}$ induces the inverse map of $\varphi|_{\ker(T_\alpha)}$. 
Assuming $x \in \ker(T_\alpha)$, we have
 \begin{align*}
  \varphi^{-1}|_{\ker(T)}(\varphi(x)) &= T(\alpha)^{-1} \sum_{k=0}^{n-1} \sum_{j=0}^k \alpha^{q^j} \left(x^{q^{k+1}} - x^{q^k}\right)\\
  &=
  T(\alpha)^{-1}\left( \sum_{k=1}^n \sum_{j=0}^{k-1}\alpha^{q^j} x^{q^k} - \sum_{k=0}^{n-1}\sum_{j=0}^{k}\alpha^{q^j}x^{q^k} \right)\\
  &=
  \left(1 - T(\alpha)^{-1} \alpha \right) x + T(\alpha)^{-1} \sum_{k=1}^{n-1}\left( \sum_{j=0}^{k-1}\alpha^{q^j} - \sum_{j=0}^{k}\alpha^{q^j} \right)x^{q^k}\\
  &= \left(1 - T(\alpha)^{-1} \alpha \right) x - T(\alpha)^{-1} \sum_{k=1}^{n-1}\left(\alpha x  \right)^{q^k} \\
  &= \left(1 - T(\alpha)^{-1} \alpha \right) x - T(\alpha)^{-1} \left( T_\alpha(x) - \alpha x \right)\\
  &= x - T(\alpha)^{-1}T_\alpha(x)\\
  &=
  x,
 \end{align*}
 as required.
\end{proof}

\begin{rmk}
 We used Theorem \ref{thm:LinearizedInverse} to obtain the coefficients $\bar{c} = (c_0 \ c_1 \cdots c_{n-1})$ of $\varphi^{-1}|_{\ker(T_\alpha)}$. 
 That is, we obtained a solution, $\bar{c}$, to the linear equation 
 $$
 \bar{c} D_\varphi = v_{q,n}(\id - T(\alpha)^{-1} T_\alpha) = -T(\alpha)^{-1}\left( \alpha - T(\alpha), \alpha^q, \alpha^{q^2}, \ldots , \alpha^{q^{n-1}}   \right),
 $$
 where $T(\alpha)^{-1}T_\alpha$ is idempotent with kernel $\ker(T_\alpha)$, and $D_\varphi$ is the $n \times n$ associate Dickson matrix of $\varphi(x) = x^q - x$ given by
 $$
 D_{\varphi} = \left( \begin{array}{cccccc}
                 -1 & 1 & 0 & 0 &\cdots & 0\\
                 0 & -1 & 1 & 0 & \cdots & 0\\
                 \vdots & \vdots &\vdots &\vdots & & \vdots\\
                1 & 0 & 0 & 0 & \cdots & -1
                \end{array} \right).
 $$
\end{rmk}

The following lemma gives the compositional inverse of a class of linearized permutation
  polynomials generalizing that whose inverse was recently obtained in Theorem 2.2, \cite{wu_new}.
  See Corollary \ref{cor: wu's result}.
  The method utilized here to obtain such a result, as an application of Theorem \ref{inv0}, seems much less complicated than that employed in \cite{wu_new}.
  
\begin{lem}\label{lem: simple proof}
 Let $\alpha \in \mathbb{F}_{q^n}$, and define the polynomial
 $$
 f(x) := x^q - x + T\left(\alpha x \right) \in \Fn[x].
 $$
 Then the following two results hold.
 
 (a) $f$ is a permutation polynomial over $\Fn$ if and only if $T(\alpha) \neq 0$, and the characteristic, $p$, of $\F$, does not divide $n$.
 
 (b) If $T(\alpha) \neq 0$ and $p \nmid n$ (equivalently, $f$ is a permutation polynomial over $\Fn$), then the compositional inverse of $f$ over $\Fn$ is
 $$
 f^{-1}(x) = T(\alpha)^{-1}n^{-1}\left( T(x) + B(x) \right),
 $$
 where the coefficients of $B(x) = \sum_{k=0}^{n-1}b_k x^{q^k} \in \Fn[x]$ are given by
 $$
 b_k = \sum_{j=1}^{n-1}j \alpha^{q^j} - n \sum_{l = k+1}^{n-1} \alpha^{q^l}, \ \ \ \ 0 \leq k \leq n-1.
 $$
\end{lem}

\begin{proof}
 (a) As done in Lemma \ref{lem: special case}, denote $\varphi(x) := x^q - x$, $T_\alpha(x) := T(\alpha x)$.
 The reader can check that $\varphi \circ T_\alpha = T \circ \varphi = 0$. 
 From the proof of Lemma \ref{lem: special case} we know that $\im(T_\alpha) = \im(T) = \F$.
 Thus letting $\psi = T_\alpha$, $\bar{\psi} = T$, $g(x) = x$, $h = 1$, and $\varphi$ as before, we see that the polynomial, $f$,
 is an instance of Theorem \ref{thm0}. Hence, the permutability of $f$ boils down to whether or not both conditions (i), (ii), of Theorem \ref{thm0}, are satisfied.
 From the proof of Lemma \ref{lem: special case}, we know that $\ker(\varphi) \cap \ker(T_\alpha) = \{0\}$ if and only if $T(\alpha) \neq 0$.
 Moreover, for all $y \in \F$, we have $\bar{f}(y) = \varphi(y) + T(y) = ny$; thus $\bar{f}$ permutes $\im(T_\alpha) = \im(T) = \F$ if and only if
 $p \nmid n$. Hence, by Theorem \ref{thm0}, $f$ is a permutation polynomial if and only if $T(\alpha) \neq 0$ and $p \nmid n$.
 
 (b) We wish to apply Theorem \ref{inv0}. However in our specific case one can show that both $S_{T_\alpha}, S_T = \Fn$ which $\varphi$ does not permute. 
 As we shall see, we can bypass this problem by considering instead the permutation polynomial $H := T(\alpha)^{-1} f$. 
 Noting that $ \varphi \circ T(\alpha)^{-1}T_\alpha = n^{-1}T \circ \varphi = 0$ and $\im(T(\alpha)^{-1}T_\alpha) = \im(n^{-1}T) = \F$,
 we see that $H$ is an instance of Theorem \ref{thm0} with $\psi = T(\alpha)^{-1}T_\alpha$, $\bar{\psi} = n^{-1}T$, $h = T(\alpha)^{-1}$, 
 and $\varphi,g,$ as before. We have, for all $y \in \F$, $\bar{H}(y) = \varphi(y) + n^{-1}T(y) = y$, thus permuting $\F = \im(\psi) = \im(\bar{\psi})$.
 The reader can check that both $\psi = T(\alpha)^{-1}T_\alpha$, $\bar{\psi} = n^{-1}T$, are idempotent. Then $S_\psi = \ker(\psi) = \ker(T_\alpha)$ and 
 $S_{\bar{\psi}} = \ker(\bar{\psi}) = \ker(T)$; hence $|S_\psi| = |S_{\bar{\psi}}|$ (since $|\ker(T_\alpha)| = |\ker(T)|$ because the images of $T_\alpha, T$, are equal).
 Moreover, from the proof of (a) we know that $\ker(\varphi) \cap S_\psi = \ker(\varphi) \cap \ker(T_\alpha) = \{0\}$. 
 In particular $\ker(\varphi) \cap \psi(S_\psi) = \{0\}$ since $\psi(S_\psi) \subseteq S_\psi$.
 Thus $H$ satisfies the assumptions of Theorem \ref{inv0}, which, after the appropriate substitutions and subsequent simplification, yields
 $$
 H^{-1}(x) = n^{-1}T(x) + \varphi^{-1}|_{\ker(T)}\left(T(\alpha) \left( x - n^{-1}T(x)  \right)   \right).
 $$
 Now using the fact that $f^{-1}(x) = H^{-1}(T(\alpha)^{-1}x)$, we get
 $$
 f^{-1}(x) = T(\alpha)^{-1}n^{-1}T(x) + \varphi^{-1}|_{\ker(T)}\left( x - n^{-1}T(x)  \right).
 $$
 From Lemma \ref{lem: special case}, we may take
 $
\varphi^{-1}|_{\ker(T)}(x) = T(\alpha)^{-1} \sum_{k=0}^{n-1} \sum_{j=0}^k \alpha^{q^j} x^{q^k}.
$
Thus,
\begin{align*}
 \varphi^{-1}|_{\ker(T)}(x - n^{-1}T(x)) 
 &= 
 T(\alpha)^{-1} \sum_{k=0}^{n-1} \sum_{j=0}^k \alpha^{q^j} \left(x - n^{-1}T(x)\right)^{q^k}\\
 &= 
 T(\alpha)^{-1} \left( \sum_{k=0}^{n-1} \sum_{j=0}^k \alpha^{q^j} x^{q^k} - n^{-1} \sum_{l=0}^{n-1} \sum_{k=0}^{n-1} \sum_{j=0}^{k} \alpha^{q^j}x^{q^l} \right)\\
 &= 
 T(\alpha)^{-1} \left( \sum_{k=0}^{n-1} \sum_{j=0}^k \alpha^{q^j} x^{q^k} - n^{-1}\sum_{l=0}^{n-1} \sum_{j=0}^{n-1}(n-j)\alpha^{q^j} x^{q^l} \right)\\
 &=
 T(\alpha)^{-1} \left( \sum_{k=0}^{n-1} \sum_{j=0}^k \alpha^{q^j} x^{q^k} - n^{-1}\sum_{l=0}^{n-1}   \left( nT(\alpha) - \sum_{j=0}^{n-1} j \alpha^{q^j}  \right)  x^{q^l} \right)\\
&=
T(\alpha)^{-1}\sum_{k=0}^{n-1} \left(  \sum_{j=0}^k \alpha^{q^j} - T(\alpha) + n^{-1} \sum_{j=0}^{n-1} j \alpha^{q^j}       \right)x^{q^k}\\
&=
T(\alpha)^{-1}\sum_{k=0}^{n-1} \left( - \sum_{j=k+1}^{n-1} \alpha^{q^j} + n^{-1}\sum_{j=0}^{n-1} j \alpha^{q^j}       \right)x^{q^k}\\
&=
 T(\alpha)^{-1}n^{-1} B(x).
 \end{align*}
The result now follows immediately.
 \end{proof}

 \begin{rmk}
  To further convince the reader of the correctness of the result in (b), we give a second, more direct, proof. That is, we show $f^{-1}(f(x)) = x$.
  Since $B$ is a $q$-polynomial, it follows that
  $$
  f^{-1}(f(x)) = T(\alpha)^{-1}\left(T(\alpha x) + n^{-1}\left(B(x^q) - B(x) + B(1) T(\alpha x)\right) \right).
  $$
  We have 
  \begin{align*}
  n^{-1}B(1) 
  &= 
  \sum_{k=0}^{n-1}\left( n^{-1}\sum_{j=1}^{n-1}j \alpha ^{q^j} - \sum_{j = k+1}^{n-1} \alpha^{q^j}  \right) 
  =  \sum_{j=1}^{n-1}j \alpha ^{q^j} - \sum_{k=0}^{n-1} \sum_{j = k+1}^{n-1} \alpha^{q^j} \\
  &=   \sum_{j=1}^{n-1}j \alpha ^{q^j} - \sum_{k=0}^{n-1} \left( T(\alpha) - \sum_{j=0}^k \alpha^{q^j} \right) 
  = 
    \sum_{j=1}^{n-1}j \alpha ^{q^j} - nT(\alpha) + \sum_{k=0}^{n-1}\sum_{j=0}^k \alpha^{q^j} \\
  &=
   \sum_{j=1}^{n-1}j \alpha ^{q^j} - nT(\alpha) + \sum_{j=0}^{n-1}(n-j) \alpha^{q^j}  \\
  &=
  0.
  \end{align*}
  Additionally,
  \begin{align*}
  n^{-1} \left(B(x^q) - B(x) \right)
   &=
   n^{-1}\left(\sum_{k=1}^{n} b_{k-1}x^{q^k} - \sum_{k=0}^{n-1} b_k x^{q^k} \right)   
  =
  n^{-1}\left(\left( b_{n-1} - b_0   \right)x + \sum_{k=1}^{n-1}\left( b_{k-1} - b_k  \right)x^{q^k} \right)\\
   &= 
  \sum_{j=1}^{n-1} \alpha^{q^j}x + \sum_{k=1}^{n-1}\left( -\sum_{j=k}^{n-1} \alpha^{q^j} + \sum_{j = k+1}^{n-1} \alpha^{q^j}  \right)x^{q^k}
   =  \sum_{j=1}^{n-1} \alpha^{q^j}x - \sum_{k=1}^{n-1}(\alpha x)^{q^k}\\
   &=
    \left( T(\alpha) - \alpha  \right)x - \left(T(\alpha x) - \alpha x\right) \\
   &=
   T(\alpha)x - T(\alpha x).
   \end{align*}
   Now the result, $f^{-1}(f(x)) = x$, follows soon.
 \end{rmk}

Note in Theorem \ref{thm: simple proof} that if we let $G(x) = x$ and $c = 1$, we obtain
Lemma \ref{lem: simple proof}. We are now ready to prove Theorem \ref{thm: simple proof}. 

\begin{proof}[{\bf Proof of Theorem \ref{thm: simple proof}}]
If $T(\alpha) = 0$, then $F \circ T = F(0)$ implying $F$ is not injective. If $p \mid n$,
then $T \circ F = 0$ implying $F$ is not surjective. As a result, if $F$ permutes $\Fn$, necessarily $T(\alpha) \neq 0$ and $p \nmid n$.
Assume $T(\alpha) \neq 0$ and $p \nmid n$. We have that $f(x) = x^q - x + T(\alpha x)$ from Lemma \ref{lem: simple proof} (b) permutes $\Fn$, and
$f \circ T_\alpha = T(\alpha) n^{-1} T \circ f = T(\alpha) T_\alpha$. 
Moreover, $\im(T_\alpha) = \im( T(\alpha) n^{-1} T) = \F$,
and $T(\alpha) n^{-1} T \circ Q \circ G = 0$, where $Q(x) := x^q - x$ (since $T \circ Q = 0$).
Then letting $\varphi = f$, $\psi = T_\alpha$, $\bar{\psi} =  T(\alpha) n^{-1} T$, $g = Q \circ G$, and $h = c$,  Corollary \ref{zero} implies
that $F$ permutes $\Fn$. It follows that $F$ permutes $\Fn$ if and only if $T(\alpha) \neq 0$ and $p \nmid n$. 
In this case, Corollary \ref{zero} gives
\begin{align*}
F^{-1}(x) 
&= 
c^{-1} f^{-1} \left( x - Q \circ G \circ c^{-1} f^{-1} \circ T(\alpha) n^{-1}T(x)   \right)\\
&= c^{-1} f^{-1} (x) - c^{-1}f^{-1} \circ Q \circ G \circ c^{-1} f^{-1} \circ T(\alpha) n^{-1}T(x).
\end{align*}
From the proof of Lemma \ref{lem: simple proof} (b), we know that
$$
f^{-1}(x) = T(\alpha)^{-1}n^{-1}T(x) + Q^{-1}|_{\ker(T)}\left( x - n^{-1}T(x)  \right),
$$
where $Q^{-1}|_{\ker(T)}$ is a $q$-polynomial inducing the inverse map of $Q|_{\ker(T_\alpha)}$. It follows that
\begin{align*}
f^{-1} \circ T(\alpha) n^{-1}T(x) 
&=
n^{-1}T(x) + Q^{-1}|_{\ker(T)}(0)\\
&= n^{-1}T(x).
\end{align*}
Hence,
$$
F^{-1} = c^{-1} f^{-1}  - c^{-1}f^{-1} \circ Q \circ G \circ c^{-1} n^{-1}T.
$$
Using the fact that $T \circ Q = 0$, we obtain
$$
f^{-1} \circ Q \circ G \circ c^{-1} n^{-1}T =  Q^{-1}|_{\ker(T)} \circ Q \circ G \circ c^{-1} n^{-1}T.
$$
Since for any $x \in \Fn$ we have $x - T(\alpha)^{-1}T(\alpha x) \in \ker(T_\alpha)$ and 
$Q(x - T(\alpha)^{-1}T(\alpha x)) = Q(x) \in \ker(T)$, we get $Q^{-1}|_{\ker(T)}(Q(x)) = x - T(\alpha)^{-1}T(\alpha x)$, for any $x \in \Fn$.
Then substituting $x$ with $G ( c^{-1} n^{-1}T(x))$ we finally obtain
\begin{align*}
F^{-1}(x) 
&= 
c^{-1} f^{-1}(x)  - c^{-1} Q^{-1}|_{\ker(T)} \circ Q \circ G \circ c^{-1} n^{-1}T(x)\\
&= c^{-1} T(\alpha)^{-1} n^{-1} \left(T(x) + B(x)   \right)  -c^{-1} \left( G(c^{-1}n^{-1}T(x)) - T(\alpha)^{-1}T\left(\alpha G(c^{-1}n^{-1}T(x))   \right)   \right)\\
&=
c^{-1} \biggr[ T(\alpha)^{-1} n^{-1} \left(T(x) + B(x)   \right) -  G\left(c^{-1}n^{-1}T(x)\right) +  T(\alpha)^{-1}T\left(\alpha G\left(c^{-1}n^{-1}T(x) \right) \right) \biggr]
\end{align*}
as required.
\end{proof}

\begin{cor}[{\bf Theorem 2.2, \cite{wu_new}}]\label{cor: wu's result}
 Let $n$ be an odd positive integer, and let $a \in \mathbb{F}_{2^n}^*$ such that $T_{2^n|2}(a^{-1}) = 1$. Then 
 $$
 f(x) = x^2 + x + T_{2^n|2}\left(\dfrac{x}{a}\right)
 $$
 is a permutation polynomial over $\mathbb{F}_{2^n}$ with compositional inverse 
 $$
 f^{-1}(x) = T_{2^n|2}(x) + B(x),
 $$
 where the coefficients of $B(x) = \sum_{k=0}^{n-1} b_k x^{2^k} \in \mathbb{F}_{2^n}[x]$ are given by
 \begin{align*}
  b_0 &= a^{-2^2} + a^{-2^4} + \cdots + a^{-2^{n-1}},\\
  b_k &= 
  \begin{cases}
  \left( a^{-2} + a^{-2^3} + \cdots + a^{-2^k}  \right) + \left( a^{-2^{k+1}} + a^{-2^{k+3}} + \cdots + a^{-2^{n-1}}    \right) & \mbox{if } k \mbox{ is odd,}\\
  \left( a^{-2} + a^{-2^3} + \cdots + a^{-2^{k-1}}   \right) + \left( a^{-2^{k+2}} + a^{-2^{k+4}} + \cdots + a^{-2^{n-1}}    \right) & \mbox{if } k \mbox{ is even,}
   \end{cases}
  \end{align*}
$1 \leq k \leq n-1$.
\end{cor}

\begin{proof}
First note that we can let $G(x) = x$ and $c = 1$ in Theorem \ref{thm: simple proof} in order to obtain Lemma \ref{lem: simple proof}.
This is Lemma \ref{lem: simple proof} with $q = 2$, $n$ odd, 
and $\alpha = a^{-1} \in \Fn$ such that $T(\alpha) = 1$. 
Since $2 \nmid n$ and $T(\alpha) \neq 0$, $f$ is a permutation polynomial by Lemma \ref{lem: simple proof} (a),
whereas the expression for the compositional inverse of $f$ follows from (b). Indeed, in our case of $q = 2$ and $n$ odd, the coefficients of $B$ are given by
\begin{align*}
b_k &:= \sum_{j=1}^{n-1}j \alpha^{q^j} - n \sum_{l = k+1}^{n-1} \alpha^{q^l}
=
\sum_{j=1}^{k} j \alpha^{2^j} + \sum_{l = k+1}^{n-1}(l + 1)\alpha^{2^l}\\
&= 
\sum_{\substack{j=1\\ j \text{ odd}   } }^{k} \alpha^{2^j} + \sum_{\substack{l = k+1\\ l \text{ even}}}^{n-1}\alpha^{2^l},
\end{align*}
where  $0 \leq k \leq n-1$.
Now the reader can check that the result is obtained by substituting $\alpha$ with $a^{-1}$ and computing the three cases of $b_k$. 
\end{proof}

 \section{Conclusion}
 In this paper we extended the decomposition method used in \cite{wu} to find the compositional inverses, and preimages in certain cases, of several more general classes of permutation polynomials over arbitrary finite fields. As a result 
 we showed how the inverses of these classes can be written in terms of the inverses, over subspaces, of two other polynomials, $\bar{f}, \varphi$, 
 where $\varphi$ is linearized.
 In some of these cases one is able to obtain both such inverses, thus obtaining the full explicit result. 
 We also showed that by solving a system of linear equations we can obtain a linearized polynomial inducing the inverse map over subspaces 
 on which a prescribed linearized polynomial induces a bijection.
 In the special case that the Dickson matrix in such a linear equation is a circulant matrix, and the characteristic $p$ of $\Fn$ does not divide $n$, the system may be solved quickly
 by writing said equation as a convolution and then applying an FFT.
 In addition, we gave the compositional inverse of a class of permutation polynomials generalizing that in (2) and (3) of the Introduction, 
 the inverses of which were obtained in \cite{coulter1} and \cite{wu}, respectively.
 We also obtained the explicit compositional inverse of a class of permutation polynomials generalizing that whose inverse was recently
 given in \cite{wu_new}. However, in many of the results we imposed the stronger condition that $\varphi$ bijects the subspace $S_\psi$, which is not a necessary condition
 for a given polynomial $f$ to induce a permutation of $\Fn$. It would thus seem to be of interest to ``ease'' this requirement, or possibly get rid of it altogether by using 
 a certain bijection $\phi : \Fn \rightarrow \psi(\Fn) \oplus \ker(\psi)$, instead of our injective map $\phi_\psi : \Fn \rightarrow \psi(\Fn) \oplus S_\psi$.
 The reason for the preference of $\ker(\psi)$ over $S_\psi$ is that in most of the cases of the permutation $f$, 
 a necessary condition for it to be a permutation is that $\varphi$ is injective on $\ker(\psi)$. 
 
 In summary, the method first used in \cite{wu} and extended here seems quite applicable, 
 and it is expected that it can similarly be applied to several more classes of permutation polynomials. The main drawback seems to be that the inverse results are written
 in terms of the inverses of $\bar{f},\varphi$, over subspaces, which one may not always be able to easily determine explicitly.


\end{document}